\documentclass[11pt,reqno]{amsart}

\usepackage{graphicx}
\usepackage{mathtools}
\usepackage{amssymb}
\usepackage{tikz}
\usepackage{braids}
 \usepackage{standalone}
\usepackage{amsmath}
\usepackage{amssymb}
\usepackage{pictexwd, dcpic}
\usepackage{caption}
\usepackage{graphicx}
\usepackage{array,subfigure}
\usepackage{multirow}
\usepackage[all]{xy}
\input xy
\xyoption{all}
 \usetikzlibrary{er,positioning, calc, cd}
\usepackage[colorlinks=true, pdfpagemode=none, pdfmenubar=false, pdfstartview=FitH, linkcolor=blue, citecolor=blue, urlcolor=blue, pdffitwindow=false]{hyperref}

\theoremstyle{plain}
\newtheorem{theorem}{Theorem}[section]

\newtheorem{lemma}[theorem]{Lemma}
\newtheorem{proposition}[theorem]{Proposition}

\newtheorem{remark}[theorem]{Remark}

\theoremstyle{definition}
\newtheorem{definition}[theorem]{Definition}

\theoremstyle{remark}

\DeclareMathOperator{\Fl}{Fl}

\DeclareMathOperator{\picr}{Pic}
\DeclareMathOperator{\Fla}{Fl}
\DeclareMathOperator{\cota}{T^*}
\DeclareMathOperator{\id}{Id}

\def\kgc_{K^*_G(\mathbb{C}^n)}
\def\kgchi_{K^*_\chi(\mathbb{C}^n)}
\def\kgcf_{K_G(\mathbb{C}^n)}
\def\kgchif_{K_\chi(\mathbb{C}^n)}
\def\gpic_{G\text{-}\picr}
\def\gcl_{G\text{-}\cl}
\def\trch_{{\chi_{0}}}

\def\genpx_{{p_X}}
\def\genpy_{{p_Y}}
\def\genpcn_{p_{\mathbb{C}^n}}

\def\C{{\mathcal{C}}}

\def\degzero{{\text{deg.0}}}

\def\degminusone{{\text{deg.-$1$}}}

\def\crossn{\bar}

\def\brTaaaCa{{T_{\crossn{1}\crossn{1}1}^{111}}}
\def\brTaaaCb{{T_{1\crossn{1}\crossn{1}}^{111}}}
\def\brDaaaCa{{D_{\crossn{1}\crossn{1}1}^{111}}}
\def\brDaaaCb{{D_{1\crossn{1}\crossn{1}}^{111}}}

\DeclareMathOperator{\cl}{Cl}

\DeclareMathOperator{\spec}{Spec\;}

\DeclareMathOperator{\trace}{tr}

\DeclareMathOperator{\action}{act}

\DeclareMathOperator{\rder}{\bf R}

\DeclareMathOperator{\cone}{Cone}

\DeclareMathOperator{\braidgp}{Br}
\DeclareMathOperator{\gbrcat}{\mathcal{G}\mathcal{B}\emph{r}}
\DeclareMathOperator{\cat}{\mathcal{C}\emph{a}\emph{t}}

\let\amsamp=&

\usepackage{todonotes}

\makeatletter    
\def\l@subsection{\@tocline{1}{0,2pt}{2pc}{8mm}{\ \ }} 
\makeatletter    
\def\l@section{\@tocline{1}{0,2pt}{2pc}{8mm}{\ \ }} 

\author{Lorenzo De Biase}
\address{ENEA Centro ricerche Bologna\\ 40129, Italy}
\email[L.~De Biase]{debiase.lorenzo7@gmail.com}

\@namedef{subjclassname@2022}{\textup{2022} Mathematics Subject Classification}
\subjclass[2020]{Primary 18E30, 14F05; Secondary 14M15, 18A05}
 
\title[Categorical actions of generalised braids for n=3]{Categorical actions of generalised braids for n=3}

\begin{document}
\begin{abstract}

Khovanov and Thomas constructed a categorical action of the braid group $\braidgp_n$ on the derived category $D(T^* \Fl_n)$ of coherent sheaves on the cotangent bundle of the variety $\Fl_n$ of the complete flags in $\mathbb{C}^n$.

In this paper, we define the generalised braid category $\gbrcat_3$, give a sufficient condition for the existence of a skein-triangulated $\gbrcat_3$ categorical action, and construct a categorical action of $\gbrcat_3$ on $D(T^*(\Fl_3(\bar{i}))$ that generalises Khovanov and Thomas categorical braid action on $D(T^* \Fl_3)$. 
\end{abstract}
\maketitle
\section{Introduction}

Braids group actions on triangulated categories have been investigated for years with many notable results.

In particular in \cite{KT}, Khovanov and Thomas constructed a categorical action of the braid group $\braidgp_n$ on the derived category $D(T^* \Fl_n)$ of coherent sheaves on the cotangent bundle of the variety $\Fl_n$ of the complete flags in $\mathbb{C}^n$.

A (weak) categorical group action of a group $G$ on $D^b(X)$, the bounded derived category of the abelian category $Coh(X)$ of coherent sheaves on a smooth quasi-projective variety $X$, is an assignment of an autoequivalence $F_g$ of $D^b(X)$ to every element $g\in G$ such that the group operation is compatible, up to isomorphism, with the composition of functors.

When $X\cong T^* \Fl_n$,  for every $i\in \{1, \dots , n \} $ the natural projection 
\[ p_i : \Fl_n \to \Fl_n (\hat{i} ) \]
 is a $\mathbb{P}^1$-bundle, where with  $\Fl_n (\hat{i} )$ we denote the flag variety where the choice of vector space of dimension $i$ is skipped. 
 
 Also the natural maps $j_i$ and $\pi_i$

\begin{equation*}\label{Horya1}
\begindc{\commdiag}[150]
\obj(1,1)[a]{$T^* \Fl_n (\hat i )$}
\obj(1,4)[b]{$p_i ^*T^* \Fl_n (\hat i )$}
\obj(6,4)[c]{$T^* \Fl_n$}
\mor{b}{a}{$\pi _i$}[0,8]
\mor{b}{c}{$j_{i}$}[0,6]
\enddc
\end{equation*}
are respectively a divisorial inclusion and a canonical projection.

Define the derived functors  
\begin{equation} \label{split}
    F_i = j_{i*} \circ  p_i^*: D^b(T^* \Fl_n (\hat i )) \to D^b (T^* \Fl_n) 
\end{equation}

and denote their right adjoints by
\[ R_i = p_{i*} \circ  j_i^!: D^b(T^* \Fl_n) \to D^b (T^* \Fl_n (\hat i )). \]

Define moreover for every $i\in \{1, \dots , n \} $ the derived functors
\begin{equation}\label{twist} T_i=\mathrm{Cone}(F_i R_i\overset{ \varepsilon _i } {\longrightarrow} \id) \qquad  T_i'=\mathrm{Cone}(\id \overset{ \epsilon _i } {\longrightarrow} F_i R_i)
\end{equation}
to be the cones (taken in appropriate enhancements) of respectively the units and the counits of the adjunctions $F_i \dashv R_i$.

In this context, in \cite{KT} Khovanov and Thomas proved the following:

\begin{theorem}[\cite{KT}, Theorem 4.5]\label{KT}

The $T_i$ are autoequivalences of $\mathrm{D^b(T^* Fl}_n)$ with inverses $T_i'$, i.e.
\[ T_i' \circ T_i \cong \id \cong T_i \circ T_i' \]
Moreover the $n-1$ autoequivalences $T_i$  satisfy the braid relations:
\begin{align*}
T_i T_j \cong T_j T_i\qquad \qquad \qquad     & \text{for }|i-j|\ge 1. \\
T_i T_j T_i \cong T_j T_i T_j \quad \quad  \qquad & \text{for } |i-j|= 1.
\end{align*}
Thus, there is a categorical action of the braid group $\mathrm{Br}_n$ on $\mathrm{D^b(T^* \Fl}_n)$.
\end{theorem}

In \cite{KT} the configuration of $n$ distinct points represents the derived category of the cotangent space of complete flags in $\mathbb{C}^n$, and the cobordism between two such configurations represents an autoequivalence of this category.

In this paper, we generalise the Khovanov and Thomas result in dimension 3 to an action of the category $\gbrcat_3$ of generalised braids on the derived categories of the cotangent bundles of the varieties of complete and partial flags in $\mathbb{C}^3$. 

Braids are topological configurations of $n$ disjoint pieces of string with $n$ fixed endpoints, considered up to isotopies which keep the strands disjoint. 

Generalised braids are braids whose strands are allowed to touch in a certain way: they can join up, continue as a multiple strand, and split apart.

We follow the general approach for $\gbrcat_n$ given in \cite{AL4}, i.e. we do not distinguish any permutations of individual strands within a multiple strand --- only the multiplicity matters. In \cite{AL4} generalised braids are defined by taking trivalent coloured graphs with fixed univalent startpoints and endpoints and satisfying flow conditions,  and ribbon-embedding these graphs into the three-dimensional space. They are then considered up to isotopies. However, in this work, since we only work with $n=3$, we define $\gbrcat_3$ directly through generating diagrams and the relations between them which we call generalised braid relations. This also ensures that this paper is self-contained and independent of \cite{AL4}. 

Due to strands having multiplicity, instead of a single endpoint configuration consisting of $3$ disjoint points, they have multiple endpoint configurations corresponding to the ordered partitions of $3$. This, together with the fact that such braids are no longer necessarily invertible, implies that generalised braids form a category rather than a group or a groupoid; the objects are the partitions of $3$ and the morphisms are the generalised braids connecting them.

Our main interest lies in categorical representations of $\gbrcat_3$  i.e. we want to assign a category to each object of $\gbrcat_3$ and a functor to each morphism of  $\gbrcat_3$, in a way that is compatible with the generalised braid relations. Moreover, we are interested in a special class of such representations defined in \cite{AL4} which are called \em skein-triangulated\rm. In these, the functors representing the braids satisfy certain triangulated-categorical relations which generalise the famous skein relation used in the construction of Jones polynomial and its categorifications (\cite{Jon} and \cite{Kho}).

\subsection*{Main results}
The results of this paper are developed following the conjecture made by Anno and Logvinenko more than a decade ago that the categorical braid group action constructed by Khovanov and Thomas in \cite{KT} can be extended to a skein-triangulated representation of the category of generalised braids. Roughly, the idea is that the functors $F_i$ in \cite{KT}  correspond to $1$-$1$-forks: a splitting of a multiplicity two strand into two multiplicity one strands. Their right adjoints $R_i$ correspond to $1$-$1$-merge: a joining of two multiplicity one strands into a multiplicity two strand. The spherical twist construction in \cite{KT} which defines the crossing functor $T_i$ to be the cone of the adjunction counit $F_i R_i \rightarrow \id$ ensures that the original skein relation holds between $F_i$, $R_i$, and $T_i$. The skein-triangulatedness of a representation of $\gbrcat_n$ implies, in particular, that it can be completely generated from its fork functors using the original skein relation and its higher analogues.  

In this paper, we prove the above conjecture for the case $n = 3$ and construct a (weak) categorical action of generalised braid category $\gbrcat_3$ on the derived categories $D(T^* \Fl_3(\hat{i}))$ of the cotangent bundles of the varieties of full and partial flags in $\mathbb{C}^3$. 

Our first result is to provide, in a general context, the list of conditions on the fork functors sufficient to generate a categorical action of $\gbrcat_3$. These conditions are easier to verify than the generalised braid relations. They give a special class of skein-triangulated representations which we call the split $\mathbb{P}^2$-representations. 

Our first result can be summarised as follows:

\begin{theorem}
Any assignment of exact functors between enhanced triangulated categories 
to the fork morphisms of $\gbrcat_3$ which satisfies our list of the split $\mathbb{P}^2$-conditions induces a (split $\mathbb{P}^2$) categorical representation of $\gbrcat_3$.
\end{theorem}

In the second part of the paper, we apply this in the context of the derived categories of the cotangent bundles of the full flag space $Fl_3$ and the partial flag spaces $\mathbb{P}^2$, $\mathbb{P}^{2\vee}$, and $\spec \mathbb{C}$. Here, the $F_i$ are either split spherical functors with cotwist $[-2]$ or $\mathbb{P}^2$ functors with $H=[-2]$. They are defined as Fourier–Mukai transforms. We verify that they satisfy the above split $\mathbb{P}^2$ conditions, thus proving the following theorem:

\begin{theorem}
Functors $F_i$ of (\ref{split}) generate a split-$\mathbb{P}^2$ skein-triangulated representation of generalised braid category $\gbrcat_3$ on $D(T^* \Fl_3(\hat{i}))$.
\end{theorem}

\subsection*{Further directions}

The most direct follow-up of this paper would consist of explicit computations establishing the conjecture in dimensions 4, and possibly 5.
Ultimately, the aim is to establish the Anno-Logvinenko conjecture for arbitrary dimension: the generalised braid category $\gbrcat_n$ has a skein-triangulated representation on the varieties of complete and partial flags in $\mathbb{C}^n$. Such representation would comprise 
a network of functors between the derived categories $D(T^* \Fl_n(\hat{i}))$, some of which are well-known in geometric representation theory: the Khovanov-Thomas braid group action \cite{KT} is a subset of the endofunctors of a single node in this network. The Cautis-Kamnitzer-Licata categorical $\mathfrak{sl}_2(\mathbb{C})$-action \cite{CautisKamnitzerLicata-DerivedEquivalencesForCotangentBundlesofGrassmanniansViaStrongcategoricalSL2Actions} is a subset of the functors which exist in this network between the Grassmanians, which are just some of its nodes. 

Moreover, data of a skein-triangulated action of generalised braids on flags in $\mathbb{C}^n$ is expected to be closely related to the data of a perverse schober (perverse sheaf of triangulated categories, \cite{KS}) on the orbifold space  $[\mathfrak{h}/W]$ where $\mathfrak{h}$ is a Cartan algebra of the Lie algebra $\mathfrak{s}\mathfrak{l}_n$ and $W$ is the corresponding Weil group. Understanding the link between the two might open new ways to understand, establish, and generalise such categorical actions.

\subsection*{Plan of the paper}

In Section \ref{gen_bra_cat_sec}, we define the generalised braid category $\gbrcat_3$ in terms of its generators and their respective relations.

In Section \ref{cat_act_sec}, we recall the notion of a categorical action and prove a theorem which gives the sufficient conditions for fork generator functors to generate a categorical action of $\gbrcat_3$ which we call a split-$\mathbb{P}^2$ action. 

In Section \ref{flag_act_sec} we introduce the setup for a generalised braid action of $\gbrcat_3$, while in Section \ref{cat_act_flag_gen_sec} we describe the merge and the fork functors and compute their compositions at level of Fourier-Mukai kernels. 

Finally, in Section \ref{cat_act_flag_final_sec} we prove that the assignments of the two previous sections satisfy the conditions of Section \ref{cat_act_sec} and therefore they give a categorical action of $\gbrcat_3$ on $D(T^*(\Fl_3(\bar{i}))$.

\subsection*{Acknowledgements}

The author is indebted to Timothy Logvinenko, Rina Anno, Ed Segal and Enrico Fatighenti for many enlightening suggestions. Thanks to Marcello Bernardara, Chris Seaman, Alice Cuzzucoli, Claudio Onorati and Aurelio Carlucci for useful discussions.  
A big part of this work was supported by an EPSRC PhD Funding grant held at the School of Mathematics of Cardiff University. The author is a member of INdAM-GNSAGA.

\section{Generalised braid category}
\label{gen_bra_cat_sec}
Intuitively, generalised braids should be thought of as braids where we remove the restriction that 
the strands are not allowed to touch each other. They can now come together, continue as a strand with multiplicity, and then 
possibly split apart: 
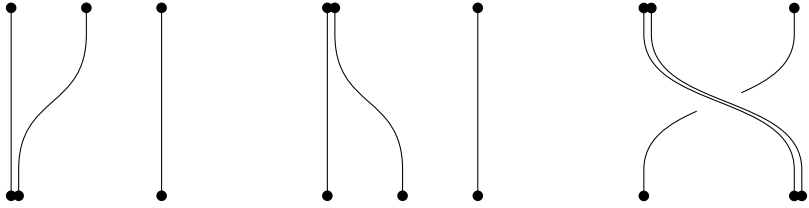
\begin{center}

\begin{tikzpicture}
\braid [number of strands=3,
style strands={1,3}{opacity=0}, height=2 cm, width=0.9cm]at (2.1,0) a_1^{-1} ;
\braid [number of strands=5,
style strands={2,4,5}{opacity=0}, height=2 cm, width=1 cm]at (2,0)  a_4;

\fill[black] (2,0) circle (2pt);
\fill[black] (2,-2.5) circle (2pt);
\fill[black] (2.1,-2.5) circle (2pt);
\fill[black] (4,0) circle (2pt);
\fill[black] (4,-2.5) circle (2pt);
\fill[black] (3,0) circle (2pt);

\end{tikzpicture}
\begin{tikzpicture}
\braid [number of strands=3,
style strands={2,3}{opacity=0}, height=2 cm, width=0.9cm]at (2.1,0) a_1 ;
\braid [number of strands=5,
style strands={2,4,5}{opacity=0}, height=2 cm, width=1 cm]at (2,0)  a_4;

\fill[black] (2,0) circle (2pt);
\fill[black] (2.1,0) circle (2pt);
\fill[black] (2,-2.5) circle (2pt);
\fill[black] (4,0) circle (2pt);
\fill[black] (4,-2.5) circle (2pt);
\fill[black] (3,-2.5) circle (2pt);

\end{tikzpicture}
\begin{tikzpicture}
\braid [number of strands=2,
style strands={3}{opacity=0}, height=2 cm, width=2 cm]at (6,-2)  a_1;
\braid [number of strands=2,
style strands={2}{opacity=0}, height=2 cm, width=2 cm]at (6.1,-2)  a_1;
\fill[black] (6,-2) circle (2pt);
\fill[black] (6.1,-2) circle (2pt);
\fill[black] (8,-2) circle (2pt);
\fill[black] (6,-4.5) circle (2pt);
\fill[black] (8.1,-4.5) circle (2pt);
\fill[black] (8,-4.5) circle (2pt);

\end{tikzpicture}
\end{center}

Any two strands with multiplicities p and q can join up and continue as a strand with multiplicity $p + q$. Any strand with multiplicity $p + q$ can split up into two strands with multiplicities $p$ and $q$. Instead of a single configuration of disjoint endpoints, we have multiple configurations indexed by ordered partitions.

We want to consider the generalised braids up to isotopies which preserve the intervals on which strands come together; they can make such interval shorter or longer, but can't make it vanish completely or join two such intervals into one. Also we do not want to distinguish individual strands within a strand with multiplicity. This approach is formalised in \cite{AL4}, where generalised braids are defined as framed $\mathbb{R}^3$-embedded trivalent coloured graphs with univalent endpoints and are considered up to isotopies of the ambient space and multifork/multimerge relations. The generalised braid category $\gbrcat_n$ is then defined accordingly. 

In this paper, to make it self-contained, we define $\gbrcat_3$ directly in terms of generating diagrams and the relations between them:

\begin{definition}
The \emph{generalised braid category} $\gbrcat_3$ is the category with:\\
\underline{Objects}: ordered partitions of $3$:
$$ \left\{ \bar{i} = (i_1 \ldots i_k) \quad \middle\vert \quad \sum_{s=1}^k \bar{i}_s =3 \right\}. $$
\underline{Morphisms}: The morphisms in $\gbrcat_3$ are called the \emph{generalised braids}. They are generated by the following set of the generating diagrams, subject to the relations further listed below:

\begin{enumerate}
    \item Four forks:
    \begin{figure}[!h] \centering 
     
            \begin{tikzpicture}
 \braid [number of strands=3,
style strands={1,3}{opacity=0}, height=1 cm, width=0.9cm]at (0.1,0) a_1^{-1};
\braid [number of strands=3,
style strands={1,3}{opacity=0}, height=1 cm, width=0.9cm]at (0.2,0) a_1^{-1};
\braid [number of strands=5,
style strands={2,3,4,5}{opacity=0}, height=1 cm, width=1 cm]at (0,0)  a_4;
\fill[black] (0,0) circle (2pt);
\fill[black] (0.2,-1.5) circle (2pt);
\fill[black] (0.1,-1.5) circle (2pt);
\fill[black] (0,-1.5) circle (2pt);
\fill[black] (1,0) circle (2pt);
\fill[black] (1.1,0) circle (2pt);

\fill[white] (-1,0) circle (2pt);
\fill[white] (2,0) circle (2pt);
\draw  node at (0.5,-2)    {$f_3^{12} $};

\braid [number of strands=3,
style strands={1,3}{opacity=0}, height=1 cm, width=0.9cm]at (1.6,0) a_2;
\braid [number of strands=3,
style strands={1,3}{opacity=0}, height=1 cm, width=0.9cm]at (1.5,0) a_2;
\braid [number of strands=5,
style strands={2,1,4,5}{opacity=0}, height=1 cm, width=1 cm]at (1.5,0)  a_4;
\fill[black] (3.5,0) circle (2pt);
\fill[black] (3.3,-1.5) circle (2pt);
\fill[black] (3.4,-1.5) circle (2pt);
\fill[black] (3.5,-1.5) circle (2pt);
\fill[black] (2.5,0) circle (2pt);
\fill[black] (2.4,0) circle (2pt);

\draw  node at (3,-2)    {$f_3^{21} $};

\braid [number of strands=3,
style strands={1,3}{opacity=0}, height=1 cm, width=1cm]at (4.9,0) a_2;
\braid [number of strands=5,
style strands={2,4,5}{opacity=0}, height=1 cm, width=1 cm]at (5,-0)  a_4;

\fill[black] (5,0) circle (2pt);
\fill[black] (7,0) circle (2pt);
\fill[black] (5,-1.5) circle (2pt);
\fill[black] (5.9,0) circle (2pt);
\fill[black] (6.9,-1.5) circle (2pt);
\fill[black] (7,-1.5) circle (2pt);

\draw  node at (6,-2)    {$f_{12}^{111} $};

  \braid [number of strands=3,
style strands={1,3}{opacity=0}, height=1 cm, width=0.9cm]at (8.6,0) a_1^{-1};
\braid [number of strands=5,
style strands={2,4,5}{opacity=0}, height=1 cm, width=1 cm]at (8.5,0)  a_4;

\fill[black] (8.5,0) circle (2pt);
\fill[black] (10.5,0) circle (2pt);
\fill[black] (8.6,-1.5) circle (2pt);
\fill[black] (8.5,-1.5) circle (2pt);
\fill[black] (9.5,0) circle (2pt);
\fill[black] (10.5,-1.5) circle (2pt);

\draw  node at (9.5,-2)    {$f_{21}^{111} $};
            \end{tikzpicture}

    \caption{Forks} \label{gbr3-forks}
    \end{figure}
    
    \item Four merges:
    \begin{figure}[!h] \centering 
        
\begin{tikzpicture}
\braid [number of strands=3,
style strands={2,3}{opacity=0}, height=1 cm, width=0.9cm]at (0.1,0) a_1;
\braid [number of strands=3,
style strands={2,3}{opacity=0}, height=1 cm, width=0.9cm]at (0.2,0) a_1;
\braid [number of strands=5,
style strands={2,3,4,5}{opacity=0}, height=1 cm, width=1 cm]at (0,0)  a_4;

\fill[white] (-1,0) circle (2pt);

\fill[black] (0,0) circle (2pt);
\fill[black] (0.1,0) circle (2pt);
\fill[black] (0.2,0) circle (2pt);
\fill[black] (0,-1.5) circle (2pt);
\fill[black] (1,-1.5) circle (2pt);
\fill[black] (1.1,-1.5) circle (2pt);

\draw  node at (0.5,-2)    {$g_{12}^{3} $};

\braid [number of strands=3,
style strands={1,2}{opacity=0}, height=1 cm, width=0.9cm]at (1.6,0) a_2^{-1};
\braid [number of strands=3,
style strands={1,2}{opacity=0}, height=1 cm, width=0.9cm]at (1.5,0) a_2^{-1};
\braid [number of strands=5,
style strands={1,2,4,5}{opacity=0}, height=1 cm, width=1 cm]at (1.5,0)  a_4^{-1};

\fill[black] (3.3,0) circle (2pt);
\fill[black] (3.4,0) circle (2pt);
\fill[black] (3.5,0) circle (2pt);

\fill[black] (2.4,-1.5) circle (2pt);
\fill[black] (2.5,-1.5) circle (2pt);
\fill[black] (3.5,-1.5) circle (2pt);

\draw  node at (3,-2)    {$g_{21}^3 $};
\braid [number of strands=3,
style strands={1,2}{opacity=0}, height=1 cm, width=0.9cm]at (5.1,0) a_2^{-1};
\braid [number of strands=5,
style strands={2,4,5}{opacity=0}, height=1 cm, width=1 cm]at (5,0)  a_4^{-1};

\fill[black] (5,0) circle (2pt);
\fill[black] (6.9,0) circle (2pt);
\fill[black] (7,0) circle (2pt);

\fill[black] (5,-1.5) circle (2pt);
\fill[black] (6,-1.5) circle (2pt);
\fill[black] (7,-1.5) circle (2pt);

\draw  node at (6,-2)    {$g_{111}^{12} $};

\braid [number of strands=3,
style strands={2,3}{opacity=0}, height=1 cm, width=0.9cm]at (8.6,0) a_1;
\braid [number of strands=5,
style strands={2,4,5}{opacity=0}, height=1 cm, width=1 cm]at (8.5,0)  a_4;

\fill[black] (8.5,0) circle (2pt);
\fill[black] (8.6,0) circle (2pt);
\fill[black] (10.5,0) circle (2pt);
\fill[black] (8.5,-1.5) circle (2pt);
\fill[black] (9.5,-1.5) circle (2pt);
\fill[black] (10.5,-1.5) circle (2pt);

\draw  node at (9.5,-2)    {$g_{111}^{21} $};
            \end{tikzpicture}

    \caption{Merges} \label{gbr3-merges}
    \end{figure}
    
    \item Two positive and two negatives $(1,1)$-crossings:
    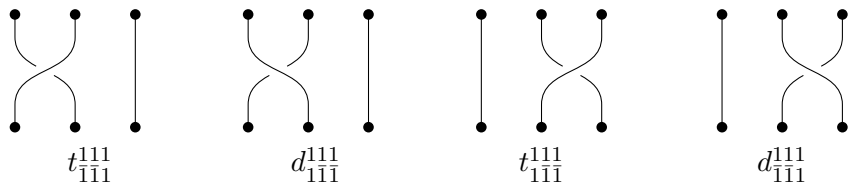
\begin{figure}[!h] \centering 
       
            \begin{tikzpicture}
\braid [number of strands=3, height=1 cm, width=0.8cm]at (0,0) a_1^{-1};

\fill[black] (0,0) circle (2pt);
\fill[black] (0.8,0) circle (2pt);
\fill[black] (1.6,0) circle (2pt);
\fill[black] (0,-1.5) circle (2pt);
\fill[black] (0.8,-1.5) circle (2pt);
\fill[black] (1.6,-1.5) circle (2pt);

\draw  node at (1,-2)    {$t_{\crossn{1}\crossn{1}1}^{111}$ };

\braid [number of strands=3, height=1 cm, width=0.8cm]at (3.1,0) a_1;

\fill[black] (3.1,0) circle (2pt);
\fill[black] (3.9,0) circle (2pt);
\fill[black] (4.7,0) circle (2pt);
\fill[black] (3.1,-1.5) circle (2pt);
\fill[black] (3.9,-1.5) circle (2pt);
\fill[black] (4.7,-1.5) circle (2pt);

\draw  node at (4,-2)    {$d_{1\crossn{1}\crossn{1}}^{111}$ };

\braid [number of strands=3, height=1 cm, width=0.8cm]at (6.2,0) a_2^{-1};

\fill[black] (6.2,0) circle (2pt);
\fill[black] (7,0) circle (2pt);
\fill[black] (7.8,0) circle (2pt);
\fill[black] (6.2,-1.5) circle (2pt);
\fill[black] (7,-1.5) circle (2pt);
\fill[black] (7.8,-1.5) circle (2pt);

\draw  node at (7,-2)    {$t_{1\crossn{1}\crossn{1}}^{111}$ };

\braid [number of strands=3, height=1 cm, width=0.8cm]at (9.4,0) a_2;

\fill[black] (9.4,0) circle (2pt);
\fill[black] (10.2,0) circle (2pt);
\fill[black] (11,0) circle (2pt);
\fill[black] (9.4,-1.5) circle (2pt);
\fill[black] (10.2,-1.5) circle (2pt);
\fill[black] (11,-1.5) circle (2pt);

\draw  node at (10.2,-2)    {${d_{\crossn{1}\crossn{1}1}^{111}}$ };

\fill[white] (12,0) circle (2pt);

            \end{tikzpicture}

    \caption{$(1,1)$-crossings} \label{gbr3-11crossings}
    \end{figure}
    
    \item Two positive and two negatives $(2,1)$- and $(1,2)$-crossings:
    \begin{figure}[!h] \centering 
        
            \begin{tikzpicture}
\braid [number of strands=2,
style strands={3}{opacity=0}, height=1.5 cm, width=2 cm]at (0,0)  a_1^{-1};
\braid [number of strands=2,
style strands={2}{opacity=0}, height=1.5 cm, width=2 cm]at (0.1,0)  a_1^{-1};
\fill[black] (0.1,0) circle (2pt);
\fill[black] (0,0) circle (2pt);
\fill[black] (2,0) circle (2pt);
\fill[black] (0,-2) circle (2pt);
\fill[black] (2.1,-2) circle (2pt);
\fill[black] (2,-2) circle (2pt);

\draw  node at (1,-2.5)    {$t^{21}_{12}$ };

\braid [number of strands=2,
style strands={1}{opacity=0}, height=1.5 cm, width=2 cm]at (3,0)  a_1^{-1};
\braid [number of strands=2,
style strands={1}{opacity=0}, height=1.5 cm, width=2 cm]at (3.1,0)  a_1^{-1};
\braid [number of strands=2,
style strands={2}{opacity=0}, height=1.5 cm, width=2.15 cm]at (3,0)  a_1^{-1};
\fill[black] (3,-2) circle (2pt);
\fill[black] (3.1,-2) circle (2pt);
\fill[black] (5.15,-2) circle (2pt);
\fill[black] (3,0) circle (2pt);
\fill[black] (5.1,0) circle (2pt);
\fill[black] (5,0) circle (2pt);
\draw  node at (4,-2.5)    {$t^{12}_{21}$ };

\braid [number of strands=2,
style strands={3}{opacity=0}, height=1.5 cm, width=2 cm]at (6,0)  a_1;
\braid [number of strands=2,
style strands={2}{opacity=0}, height=1.5 cm, width=2 cm]at (6.1,0)  a_1;
\fill[black] (6.1,0) circle (2pt);
\fill[black] (6,0) circle (2pt);
\fill[black] (8,0) circle (2pt);
\fill[black] (6,-2) circle (2pt);
\fill[black] (8.1,-2) circle (2pt);
\fill[black] (8,-2) circle (2pt);

\draw  node at (7,-2.5)    {$d_{12}^{21}$ };
\braid [number of strands=2,
style strands={1}{opacity=0}, height=1.5 cm, width=2 cm]at (9,0)  a_1;
\braid [number of strands=2,
style strands={1}{opacity=0}, height=1.5 cm, width=2 cm]at (9.1,0)  a_1;
\braid [number of strands=2,
style strands={2}{opacity=0}, height=1.5 cm, width=2.15 cm]at (9,0)  a_1;
\fill[black] (9,-2) circle (2pt);
\fill[black] (9.1,-2) circle (2pt);
\fill[black] (11.15,-2) circle (2pt);
\fill[black] (9,0) circle (2pt);
\fill[black] (11.1,0) circle (2pt);
\fill[black] (11,0) circle (2pt);

\fill[white] (12,0) circle (2pt);

\draw  node at (10,-2.5)    {$d_{21}^{12}$ };

            \end{tikzpicture}

    \caption{$(1,2)$- and $(2,1)$-crossings} \label{gbr3-12and21-crossings}
    \end{figure}
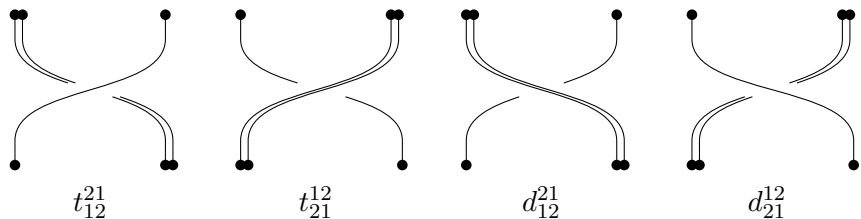
\end{enumerate}
We also denote the identity morphisms of the four objects of $\gbrcat_3$ with the following diagrams:
 \begin{figure}[!h] \centering 
     
\begin{tikzpicture}
\braid [number of strands=3, height=1 cm, width=0.1 cm]at (0,0)  1;

\fill[black] (0,0) circle (2pt);
\fill[black] (0.1,0) circle (2pt);
\fill[black] (0.2,0) circle (2pt);
\fill[black] (0,-1.5) circle (2pt);
\fill[black] (0.1,-1.5) circle (2pt);
\fill[black] (0.2,-1.5) circle (2pt);

\fill[white] (-1,0) circle (2pt);
\fill[white] (2,0) circle (2pt);
\draw  node at (0.2,-2)    {$\id_3$};

\braid [number of strands=2, height=1 cm, width=1 cm]at (2.5,0) 1;
\braid [number of strands=1, height=1 cm, width=0.1 cm]at (2.4,0) 1;

\fill[black] (2.4,0) circle (2pt);
\fill[black] (2.5,0) circle (2pt);
\fill[black] (3.5, 0) circle (2pt);
\fill[black] (2.4,-1.5) circle (2pt);
\fill[black] (2.5,-1.5) circle (2pt);
\fill[black] (3.5,-1.5) circle (2pt);

\draw  node at (3,-2)    {$\id_{21} $};

\braid [number of strands=2, height=1 cm, width=1 cm]at (5.5,0) 1;
\braid [number of strands=1, height=1 cm, width=0.1 cm]at (6.6,0) 1;

\fill[black] (5.5,0) circle (2pt);
\fill[black] (6.5,0) circle (2pt);
\fill[black] (6.6, 0) circle (2pt);
\fill[black] (5.5,-1.5) circle (2pt);
\fill[black] (6.5,-1.5) circle (2pt);
\fill[black] (6.6,-1.5) circle (2pt);

\draw  node at (6,-2)    {$\id_{12} $};

  \braid [number of strands=3, height=1 cm, width=1cm]at (8.5,0) 1;

\fill[black] (8.5,0) circle (2pt);
\fill[black] (9.5,0) circle (2pt);
\fill[black] (10.5,0) circle (2pt);
\fill[black] (8.5,-1.5) circle (2pt);
\fill[black] (9.5,-1.5) circle (2pt);
\fill[black] (10.5,-1.5) circle (2pt);

\draw  node at (9.5,-2)    {$\id_{111}$};
\end{tikzpicture}

\caption{Identity morphism diagrams} \label{gbr3-ids}
\end{figure}
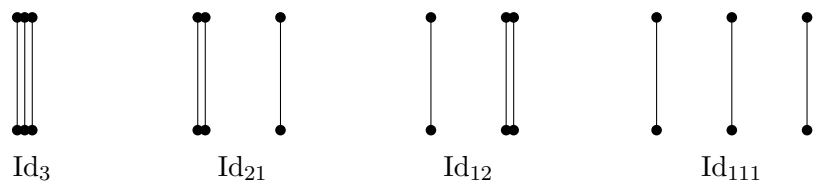

We impose the following relations on the compositions of the generating diagrams: 

\begin{enumerate}
\label{list-generalised-braid-relations-for-gbr3}
\item The multifork relation: $f_{21}^{111} f_3^{21} = f_{12}^{111} f_3^{12}$.
\begin{figure}[!h] \centering 
  \begin{tikzpicture}

\braid [number of strands=3,
style strands={1,3}{opacity=0}, height=1 cm, width=0.9cm]at (0.1,0) a_1^{-1};
\braid [number of strands=3,
style strands={1,3}{opacity=0}, height=1 cm, width=0.9cm]at (0.2,0) a_1^{-1};
\braid [number of strands=5,
style strands={2,3,4,5}{opacity=0}, height=1 cm, width=1 cm]at (0,0)  a_4;

\braid [number of strands=5,
style strands={3,4,5}{opacity=0}, height=1 cm, width=1 cm]at (0,1.5)  a_4;
\braid [number of strands=3,
style strands={1,3}{opacity=0}, height=1 cm, width=0.9cm]at (1.1,1.5) a_1^{-1};

\fill[black] (2,1.5) circle (2pt);
\fill[black] (1,1.5) circle (2pt);
\fill[black] (0,1.5) circle (2pt);
\fill[black] (0,0) circle (2pt);
\fill[black] (0.2,-1.5) circle (2pt);
\fill[black] (0.1,-1.5) circle (2pt);
\fill[black] (0,-1.5) circle (2pt);
\fill[black] (1,0) circle (2pt);
\fill[black] (1.1,0) circle (2pt);

\braid [number of strands=3,
style strands={1,3}{opacity=0}, height=1 cm, width=0.9cm]at (4.1,0) a_2;
\braid [number of strands=3,
style strands={1,3}{opacity=0}, height=1 cm, width=0.9cm]at (4,0) a_2;
\braid [number of strands=3,
style strands={1,3}{opacity=0}, height=1 cm, width=0.9cm]at (3.1,1.5) a_2;

\braid [number of strands=5,
style strands={2,1,4,5}{opacity=0}, height=1 cm, width=1 cm]at (4,0)  a_4;
\braid [number of strands=5,
style strands={1,4,5}{opacity=0}, height=1 cm, width=1 cm]at (4,1.5)  a_4;

\fill[black] (4,1.5) circle (2pt);
\fill[black] (5,1.5) circle (2pt);
\fill[black] (6,1.5) circle (2pt);
\fill[black] (6,0) circle (2pt);
\fill[black] (5.8,-1.5) circle (2pt);
\fill[black] (5.9,-1.5) circle (2pt);
\fill[black] (6,-1.5) circle (2pt);
\fill[black] (5,0) circle (2pt);
\fill[black] (4.9,0) circle (2pt);

\fill[white] (-2,0) circle (2pt);

\draw  node at (3,0)    {$ = $};

\end{tikzpicture}
\caption{The Multifork relation}
\end{figure}
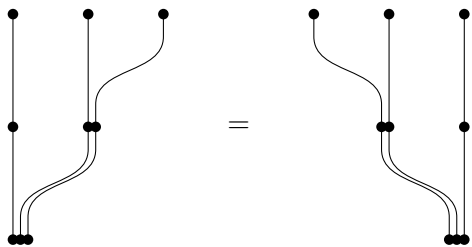

\item The braid relation: $ t_{\crossn{1}\crossn{1}1}^{111}t_{1\crossn{1}\crossn{1}}^{111}t_{\crossn{1}\crossn{1}1}^{111}=t_{1\crossn{1}\crossn{1}}^{111}t_{\crossn{1}\crossn{1}1}^{111}t_{1\crossn{1}\crossn{1}}^{111}$.

\begin{figure}[!h] \centering 
  \begin{tikzpicture}
\braid [number of strands=3, height=1 cm, width=1cm]at (0,0) a_1 a_2 a_1;

\fill[black] (0,0) circle (2pt);
\fill[black] (1,0) circle (2pt);
\fill[black] (2,0) circle (2pt);
\fill[black] (0,-1.16) circle (2pt);
\fill[black] (1,-1.16) circle (2pt);
\fill[black] (2,-1.16) circle (2pt);
\fill[black] (0,-2.32) circle (2pt);
\fill[black] (1,-2.32) circle (2pt);
\fill[black] (2,-2.32) circle (2pt);
\fill[black] (0,-3.5) circle (2pt);
\fill[black] (1,-3.5) circle (2pt);
\fill[black] (2,-3.5) circle (2pt);

\braid [number of strands=3, height=1 cm, width=1cm]at (5,0) a_2 a_1 a_2;

\fill[black] (5,0) circle (2pt);
\fill[black] (6,0) circle (2pt);
\fill[black] (7,0) circle (2pt);
\fill[black] (5,-1.16) circle (2pt);
\fill[black] (6,-1.16) circle (2pt);
\fill[black] (7,-1.16) circle (2pt);
\fill[black] (5,-2.32) circle (2pt);
\fill[black] (6,-2.32) circle (2pt);
\fill[black] (7,-2.32) circle (2pt);
\fill[black] (5,-3.5) circle (2pt);
\fill[black] (6,-3.5) circle (2pt);
\fill[black] (7,-3.5) circle (2pt);

\draw  node at (3.5,-1.75)    {$ = $};
\end{tikzpicture}
\caption{The braid relation}
\end{figure}
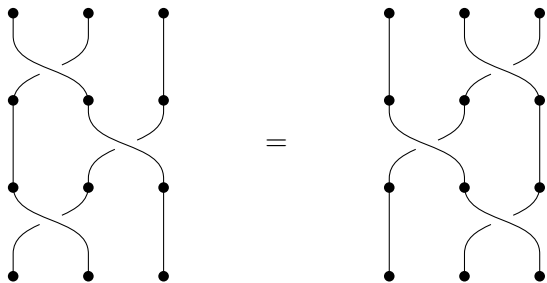

\item The inverse relations: $t_{21}^{12} d_{12}^{21} = \id_{12}, d_{12}^{21} t_{21}^{12} = \id_{21}$, $t_{\crossn{1}\crossn{1}1}^{111}d_{\crossn{1}\crossn{1}1}^{111} = t_{1\crossn{1}\crossn{1}}^{111} d_{1\crossn{1}\crossn{1}}^{111} = \id_{111}$.

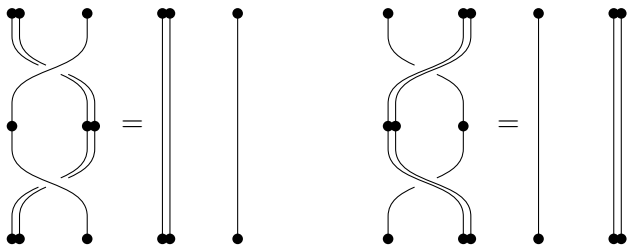
\begin{figure}[!h] \centering

\begin{tikzpicture}

\fill[white] (-2,0) circle (2pt);

\fill[black] (0,1.5) circle (2pt);
\fill[black] (0.1,1.5) circle (2pt);
\fill[black] (1,1.5) circle (2pt);

\braid [number of strands=2,
style strands={3}{opacity=0}, height=1 cm, width=1 cm]at (0,1.5)  a_1^{-1};
\braid [number of strands=2,
style strands={2}{opacity=0}, height=1 cm, width=1 cm]at (0.1,1.5)  a_1^{-1};

\fill[black] (1.1,0) circle (2pt);
\fill[black] (0,0) circle (2pt);
\fill[black] (1,0) circle (2pt);
\fill[black] (0,-1.5) circle (2pt);
\fill[black] (0.1,-1.5) circle (2pt);
\fill[black] (1,-1.5) circle (2pt);

\braid [number of strands=2,
style strands={1}{opacity=0}, height=1 cm, width=1 cm]at (0,0)  a_1;
\braid [number of strands=2,
style strands={1}{opacity=0}, height=1 cm, width=1 cm]at (0.1,0)  a_1;
\braid [number of strands=2,
style strands={2}{opacity=0}, height=1 cm, width=1 cm]at (0,0)  a_1;

\draw  node at (1.6,0)    {$ = $ };

\braid [number of strands=4,
style strands={3,4}{opacity=0}, height=2.5 cm, width=1 cm]at (2,1.5)  a_3;

\braid [number of strands=4,
style strands={2,3,4}{opacity=0}, height=2.5 cm, width=1 cm]at (2.1,1.5)  a_3;

\fill[black] (2.1,1.5) circle (2pt);
\fill[black] (2,1.5) circle (2pt);
\fill[black] (3,1.5) circle (2pt);
\fill[black] (2,-1.5) circle (2pt);
\fill[black] (2.1,-1.5) circle (2pt);
\fill[black] (3,-1.5) circle (2pt);

\fill[black] (5,1.5) circle (2pt);
\fill[black] (6.1,1.5) circle (2pt);
\fill[black] (6,1.5) circle (2pt);

\braid [number of strands=2,
style strands={3}{opacity=0}, height=1 cm, width=1 cm]at (5,0)  a_1;
\braid [number of strands=2,
style strands={2}{opacity=0}, height=1 cm, width=1 cm]at (5.1,0)  a_1;

\fill[black] (5.1,0) circle (2pt);
\fill[black] (5,0) circle (2pt);
\fill[black] (6,0) circle (2pt);
\fill[black] (5,-1.5) circle (2pt);
\fill[black] (6.1,-1.5) circle (2pt);
\fill[black] (6,-1.5) circle (2pt);

\braid [number of strands=2,
style strands={1}{opacity=0}, height=1 cm, width=1 cm]at (5,1.5)  a_1^{-1};
\braid [number of strands=2,
style strands={1}{opacity=0}, height=1 cm, width=1 cm]at (5.1,1.5)  a_1^{-1};
\braid [number of strands=2,
style strands={2}{opacity=0}, height=1 cm, width=1 cm]at (5,1.5)  a_1^{-1};

\draw  node at (6.6,0)    {$ = $ };

\braid [number of strands=4,
style strands={3,4}{opacity=0}, height=2.5 cm, width=1 cm]at (7,1.5)  a_3;

\braid [number of strands=4,
style strands={1,3,4}{opacity=0}, height=2.5 cm, width=1 cm]at (7.1,1.5)  a_3;

\fill[black] (8.1,1.5) circle (2pt);
\fill[black] (7,1.5) circle (2pt);
\fill[black] (8,1.5) circle (2pt);
\fill[black] (7,-1.5) circle (2pt);
\fill[black] (8.1,-1.5) circle (2pt);
\fill[black] (8,-1.5) circle (2pt);

\braid [number of strands=3, height=1 cm, width=0.75 cm]at (1,-3)  a_1 a_1^{-1};

\fill[black] (1,-3) circle (2pt);
\fill[black] (1.75,-3) circle (2pt);
\fill[black] (2.5,-3) circle (2pt);
\fill[black] (1,-4.25) circle (2pt);
\fill[black] (1.75,-4.25) circle (2pt);
\fill[black] (2.5,-4.25) circle (2pt);
\fill[black] (1,-5.5) circle (2pt);
\fill[black] (1.75,-5.5) circle (2pt);
\fill[black] (2.5,-5.5) circle (2pt);

\draw  node at (3,-4.25)    {$ = $ };

\braid [number of strands=5,style strands={5,4}{opacity=0}, height=2 cm, width=0.75 cm]at (3.5,-3)  a_4;

\fill[black] (3.5,-3) circle (2pt);
\fill[black] (4.25,-3) circle (2pt);
\fill[black] (5,-3) circle (2pt);

\fill[black] (3.5,-5.5) circle (2pt);
\fill[black] (4.25,-5.5) circle (2pt);
\fill[black] (5,-5.5) circle (2pt);

\draw  node at (5.5,-4.25)    {$ = $ };

\braid [number of strands=3, height=1 cm, width=0.75 cm]at (6,-3)  a_2 a_2^{-1};

\fill[black] (6,-3) circle (2pt);
\fill[black] (6.75,-3) circle (2pt);
\fill[black] (7.5,-3) circle (2pt);
\fill[black] (6,-4.25) circle (2pt);
\fill[black] (6.75,-4.25) circle (2pt);
\fill[black] (7.5,-4.25) circle (2pt);
\fill[black] (6,-5.5) circle (2pt);
\fill[black] (6.75,-5.5) circle (2pt);
\fill[black] (7.5,-5.5) circle (2pt);

\end{tikzpicture}

\caption{The inverse relations}
\end{figure}

\item The pitchfork relation: $f^{111}_{12} t^{12}_{21}= t_{\crossn{1}\crossn{1}1}^{111} t_{1\crossn{1}\crossn{1}}^{111} f^{111}_{21}$.

\begin{figure}[!h] \centering 
\begin{tikzpicture}
\fill[white] (1,0) circle (2pt);

\braid [number of strands=2,
style strands={1}{opacity=0}, height=1.5 cm, width=1.9 cm]at (3,0)  a_1;
\braid [number of strands=2,
style strands={1}{opacity=0}, height=1.5 cm, width=1.9 cm]at (3.1,0)  a_1;
\braid [number of strands=2,
style strands={2}{opacity=0}, height=1.5 cm, width=2.05 cm]at (3,0)  a_1;
\fill[black] (3,-2) circle (2pt);
\fill[black] (3.1,-2) circle (2pt);
\fill[black] (5.05,-2) circle (2pt);
\fill[black] (3,0) circle (2pt);
\fill[black] (4.9,0) circle (2pt);
\fill[black] (5,0) circle (2pt);

\braid [number of strands=3,
style strands={1,3}{opacity=0}, height=1 cm, width=1cm]at (2.9,1.5) a_2;
\braid [number of strands=5,
style strands={2,4,5}{opacity=0}, height=1 cm, width=1 cm]at (3,1.5)  a_4;

\fill[black] (3,1.5) circle (2pt);
\fill[black] (5,1.5) circle (2pt);
\fill[black] (3.9,1.5) circle (2pt);

\draw  node at (6.25,-0.25)    {$ = $};

\braid [number of strands=3, height=0.85 cm, width=1cm]at (7.5,1.5) a_1 a_2;

\braid [number of strands=3,
style strands={1,3}{opacity=0}, height=0.75 cm, width=0.9cm]at (7.6,-0.7) a_1^{-1};
\braid [number of strands=5,
style strands={2,4,5}{opacity=0}, height=0.75 cm, width=1 cm]at (7.5,-0.7)  a_4;

\fill[black] (7.5,1.5) circle (2pt);
\fill[black] (8.5,1.5) circle (2pt);
\fill[black] (9.5,1.5) circle (2pt);

\fill[black] (7.5,0.45) circle (2pt);
\fill[black] (8.5,0.45) circle (2pt);
\fill[black] (9.5,0.45) circle (2pt);

\fill[black] (7.5,-0.7) circle (2pt);
\fill[black] (9.5,-0.7) circle (2pt);
\fill[black] (7.6,-2) circle (2pt);
\fill[black] (7.5,-2) circle (2pt);
\fill[black] (8.5,-0.7) circle (2pt);
\fill[black] (9.5,-2) circle (2pt);

\end{tikzpicture}
\caption{The pitchfork relation}
\end{figure}
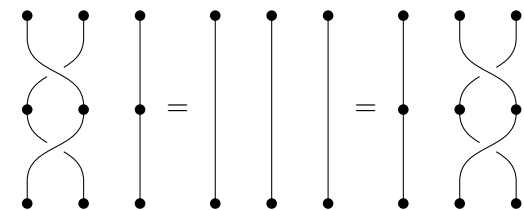

\end{enumerate}

and all the relations obtainable from these via the three operations:

\begin{itemize}
    \item \emph{vertical reflection}: swap the source and the target partitions, change all forks into merges and vice versa, reverse the parity of the crossings;
    \item \emph{horizontal reflection}: swap the partitions $12$ and $21$, reverse the parity of the crossings;
    \item \emph{blackboard reflection}: reverse the parity of the crossings. 
\end{itemize}

\end{definition}

\begin{remark}\label{gbr2}  
The category $\gbrcat_2$ has a similar description. 

Its two objects are the ordered partitions of $2$: $ (2)$  and   $(11)$. 

Its generating morphisms are $f_{2}^{11}$, $g_{11}^{2}$, ${t_{\crossn{1}\crossn{1}}^{11}}$, ${d_{\crossn{1}\crossn{1}}^{11}}$, as depicted in Figure \ref{gbr2-generators}, 
and the only relation between them is
\begin{equation}\label{gbr2relation}
 {t_{\crossn{1}\crossn{1}}^{11}} {d_{\crossn{1}\crossn{1}}^{11}} = \id_{11} = {d_{\crossn{1}\crossn{1}}^{11}} {t_{\crossn{1}\crossn{1}}^{11}}.
\end{equation}

\begin{figure}[!h] \centering 
\label{figure-gbr2-generators}
     
\begin{tikzpicture}
 \braid [number of strands=3,
style strands={1,3}{opacity=0}, height=1 cm, width=0.9cm]at (0.1,0) a_1^{-1};
\braid [number of strands=5,
style strands={2,3,4,5}{opacity=0}, height=1 cm, width=1 cm]at (0,0)  a_4;
\fill[black] (0,0) circle (2pt);
\fill[black] (0.1,-1.5) circle (2pt);
\fill[black] (0,-1.5) circle (2pt);
\fill[black] (1,0) circle (2pt);

\fill[white] (-1,0) circle (2pt);
\fill[white] (2,0) circle (2pt);
\draw  node at (0.5,-2)    {$f_2^{11} $};

\braid [number of strands=3,
style strands={2,3}{opacity=0}, height=1 cm, width=0.9cm]at (3.1,0) a_1;

\braid [number of strands=5,
style strands={2,3,4,5}{opacity=0}, height=1 cm, width=1 cm]at (3,0)  a_4;

\fill[black] (3,-1.5) circle (2pt);
\fill[black] (4,-1.5) circle (2pt);
\fill[black] (3,0) circle (2pt);
\fill[black] (3.1,0) circle (2pt);

\draw  node at (3.5,-2)    {$g_{11}^{2} $};

\braid [number of strands=1, height=1 cm, width=1cm]at (6,0) a_1;

\fill[black] (6,0) circle (2pt);
\fill[black] (7,0) circle (2pt);
\fill[black] (6,-1.5) circle (2pt);
\fill[black] (7,-1.5) circle (2pt);

\draw  node at (6.5,-2)    {${t_{\crossn{1}\crossn{1}}^{11}} $};

\braid [number of strands=2, height=1 cm, width=1cm]at (9,0) a_1^{-1};

\fill[black] (9,0) circle (2pt);
\fill[black] (10,0) circle (2pt);
\fill[black] (9,-1.5) circle (2pt);
\fill[black] (10,-1.5) circle (2pt);

\draw  node at (9.5,-2)    {${d_{\crossn{1}\crossn{1}}^{11}} $};
            \end{tikzpicture}

    \caption{Generators of $\gbrcat_2$} \label{gbr2-generators}
    \end{figure}

\end{remark}

\newpage

\section{Categorical braid actions of $\gbrcat_3$}
\label{cat_act_sec}

In this section, we focus on the \emph{categorical actions} of $\gbrcat_3$ on enhanced triangulated categories.

While at a general level the nature of such actions is clear, things become considerably more complicated when it comes to verifying whether a particular assignment of a network of categories and functors is a categorical action or not.

From this perspective, we give a result on some conditions, which are easier to verify, sufficient for a given assignment 
is a categorical action of $\gbrcat_3$.

A posteriori, following the results in \cite{AL4}, such conditions define a special case of what is called a skein-triangulated representation.

We refer to \cite{AL2} and \cite{To}   for the details on enhanced triangulated categories.

\subsection{Definitions and general notations}

Let $\bar{i}=(i_1,\ldots, i_k)$ denote a partition of $3$. The general principle is that the source partition and the target partition are denoted by the subscript and the superscript, respectively. Specifically, if $\bar{i}$ and $\bar{j}$ allow a fork between them, let $f_{\bar{i}}^{\bar{j}}$ denote this unique fork. Similarly, let $g_{\bar{i}}^{\bar{j}}$ denote the unique merge and when $\bar{i} \neq \bar{j}$, let $t_{\bar{i}}^{\bar{j}}$ and $d_{\bar{i}}^{\bar{j}}$ denote the unique positive and the negatives crossings. When $\bar{i} = \bar{j}$, there is an ambiguity, so we overscore the subscript indices which correspond to the strands being crossed, e.g. $t_{\crossn{1}\crossn{1}1}^{111}$ or $t_{1\crossn{1}\crossn{1}}^{111}$.

\begin{definition}
A \emph{(weak) categorical action of $\gbrcat_3$} is the assignment of a category to each object of $\gbrcat_3$ and a functor to each morphism of  $\gbrcat_3$, in a way that is compatible with the generalised braid relations.

In other words, consider the category $\cat /\mathcal{I}so $ the category whose objects are small categories and whose morphisms are the equivalence classes of functors between them up to isomorphism.

Then a (weak) categorical action of $\gbrcat_3$ is a functor from $\gbrcat_3$ to $\cat /\mathcal{I}so $.
\end{definition}

Next, we define a special class of categorical representations of $\gbrcat_3$ where it acts on enhanced triangulated categories by enhanced functors, and these functors satisfy several triangulated-categorical relations. In particular, these relations ensure that the whole representation is generated from its fork functors:

\begin{definition}
\label{defn-split-p^2-representation-of-gbr-3}
A split-$\mathbb{P}^2$ representation of $\gbrcat_3$ is a categorical representation of $\gbrcat_3$ which satisfies the following conditions:
\begin{enumerate}
    \item Each partition $\bar{i}$ of $3$ is represented by an enhanced triangulated category $\C_{\bar{i}}$.
    \item 
    Generators $f_{\bar{i}}^{\bar{j}}$, $g_{\bar{i}}^{\bar{j}}$, $t_{\bar{i}}^{\bar{j}}$, $d_{\bar{i}}^{\bar{j}}$ are represented by enhanced exact functors $F_{\bar{i}}^{\bar{j}}$, $G_{\bar{i}}^{\bar{j}}$, $T_{\bar{i}}^{\bar{j}}$, $D_{\bar{i}}^{\bar{j}}$. 
    \item
    \label{item-existence-of-both-adjoints}
    Each fork functor $F_{\bar{i}}^{\bar{j}}$ has enhanced left adjoint $L_{\bar{j}}^{\bar{i}}$ and enhanced right adjoint $R_{\bar{j}}^{\bar{i}}$. 
    \item 
    \label{item-1-2-and-2-1-forks-are-spherical-functors}
    $(1,2)$-fork $F_{3}^{12}$ and $(2,1)$-fork $F_{3}^{21}$ are split $\mathbb{P}^2$-functors with $H = [-2]$. 
    \item 
    \label{item-1-1-forks-are-spherical-functors}
    $(1,1)$-forks $F_{12}^{111}$ and $F_{21}^{111}$ are split spherical functors with cotwist $[-2]$.  
    
           \item 
        \label{item-fl-3-condition}
        There exist isomorphisms 
        \begin{align*}
            R_{12}^3 R_{111}^{12} F_{12}^{111} F_{3}^{12} \simeq \id_3 \oplus [-2] \oplus [-2] \oplus [-4] \oplus [-4] \oplus [-6]
            \simeq R_{21}^3 R_{111}^{21} F_{21}^{111} F_{3}^{21}
        \end{align*}
        
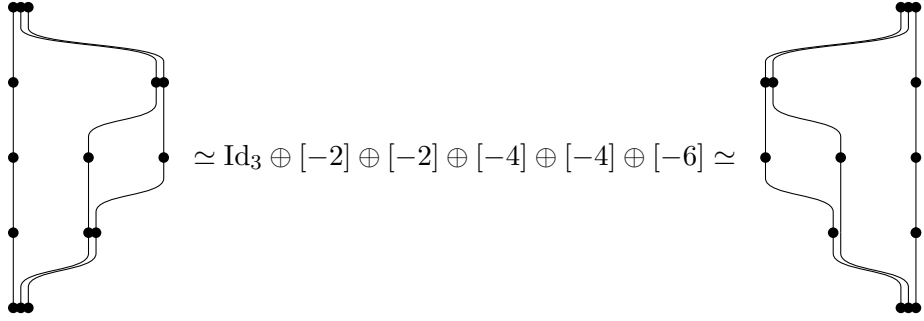
\begin{figure}[!h] \centering 
\adjustbox{scale=1,center}{
\begin{tikzpicture}

\fill[white] (-2,2) circle (2pt);

\braid [number of strands=3,
style strands={1,3}{opacity=0}, height=0.5 cm, width=0.9cm]at (0.1,1) a_1^{-1};
\braid [number of strands=3,
style strands={1,3}{opacity=0}, height=0.5 cm, width=0.9cm]at (0.2,1) a_1^{-1};
\braid [number of strands=5,
style strands={2,3,4,5}{opacity=0}, height=0.5 cm, width=1 cm]at (0,1)  a_4;

\braid [number of strands=5,
style strands={3,4,5}{opacity=0}, height=0.5 cm, width=1 cm]at (0,2)  a_4;
\braid [number of strands=3,
style strands={1,3}{opacity=0}, height=0.5 cm, width=0.9cm]at (1.1,2) a_1^{-1};

\fill[black] (2,2) circle (2pt);
\fill[black] (1,2) circle (2pt);
\fill[black] (0,2) circle (2pt);
\fill[black] (0,1) circle (2pt);
\fill[black] (0.2,0) circle (2pt);
\fill[black] (0.1,0) circle (2pt);
\fill[black] (0,0) circle (2pt);
\fill[black] (1,1) circle (2pt);
\fill[black] (1.1,1) circle (2pt);
\fill[black] (0.2,4) circle (2pt);
\fill[black] (0.1,4) circle (2pt);
\fill[black] (0,4) circle (2pt);
\fill[black] (2,3) circle (2pt);
\fill[black] (1.9,3) circle (2pt);
\fill[black] (0,3) circle (2pt);

\draw  node at (6,2)  {$ \simeq \mathrm{Id}_3 \oplus [-2] \oplus [-2] \oplus [-4] \oplus [-4] \oplus [-6] \simeq   $};

\braid [number of strands=3,
style strands={1,2}{opacity=0}, height=0.5 cm, width=0.9cm]at (0.1,3) a_2^{-1};
\braid [number of strands=5,
style strands={2,4,5}{opacity=0}, height=0.5 cm, width=1 cm]at (0,3)  a_4^{-1};

\braid [number of strands=3,
style strands={2,3}{opacity=0}, height=0.5 cm, width=1.8cm]at (0.1,4) a_1;
\braid [number of strands=3,
style strands={2,3}{opacity=0}, height=0.5 cm, width=1.8cm]at (0.2,4) a_1;
\braid [number of strands=5,
style strands={2,3,4,5}{opacity=0}, height=0.5 cm, width=2 cm]at (0,4)  a_4;

\braid [number of strands=3,
style strands={1,3}{opacity=0}, height=0.5 cm, width=0.9cm]at (10.1,1) a_2;
\braid [number of strands=3,
style strands={1,3}{opacity=0}, height=0.5 cm, width=0.9cm]at (10,1) a_2;
\braid [number of strands=5,
style strands={2,1,4,5}{opacity=0}, height=0.5 cm, width=1 cm]at (10,1)  a_4;

\braid [number of strands=3,
style strands={1,3}{opacity=0}, height=0.5 cm, width=0.9cm]at (9.1,2) a_2;
\braid [number of strands=5,
style strands={1,4,5}{opacity=0}, height=0.5 cm, width=1 cm]at (10,2)  a_4;

\braid [number of strands=3,
style strands={2,3}{opacity=0}, height=0.5 cm, width=0.9cm]at (10.1,3) a_1;
\braid [number of strands=5,
style strands={2,4,5}{opacity=0}, height=0.5 cm, width=1 cm]at (10,3)  a_4;

\braid [number of strands=3,
style strands={1,2}{opacity=0}, height=0.5 cm, width=1.8cm]at (8.2,4) a_2^{-1};
\braid [number of strands=3,
style strands={1,2}{opacity=0}, height=0.5 cm, width=1.8cm]at (8.3,4) a_2^{-1};
\braid [number of strands=5,
style strands={1,2,4,5}{opacity=0}, height=0.5 cm, width=1 cm]at (10,4)  a_4^{-1};

\fill[black] (10,2) circle (2pt);
\fill[black] (11,2) circle (2pt);
\fill[black] (12,2) circle (2pt);
\fill[black] (12,1) circle (2pt);
\fill[black] (11.8,0) circle (2pt);
\fill[black] (11.9,0) circle (2pt);
\fill[black] (12,0) circle (2pt);
\fill[black] (12,1) circle (2pt);
\fill[black] (10.9,1) circle (2pt);

\fill[black] (12,4) circle (2pt);
\fill[black] (11.9,4) circle (2pt);
\fill[black] (11.8,4) circle (2pt);

\fill[black] (12,3) circle (2pt);
\fill[black] (10,3) circle (2pt);
\fill[black] (10.1,3) circle (2pt);

\end{tikzpicture}}

\end{figure}
        
        which together with $\mathbb{P}^2$-functor structures on $F_{3}^{12}$ and $F_{3}^{21}$ identify the maps 
        \begin{align} \label{eq1}
            R_{12}^{3} F_{3}^{12} \xrightarrow{R_{12}^{3} \action F_{3}^{12}} R_{12}^3 R_{111}^{12} F_{12}^{111} F_{3}^{12}, \\ \label{eq2}
            R_{21}^{3} F_{3}^{21} \xrightarrow{R_{21}^{3} \action F_{3}^{21}} R_{21}^3 R_{111}^{21} F_{21}^{111} F_{3}^{21},
        \end{align}    
        with the maps 
        \begin{equation*}
            \id_3 \oplus [-2] \oplus [-4] \rightarrow \id_3 \oplus [-2] \oplus [-2] \oplus [-4] \oplus [-4] \oplus [-6]
        \end{equation*}
        which are the direct summand embeddings whose images are $\id_3$ $\oplus$ the first $[-2]$ $\oplus$ the first$[-4]$, and 
        $\id_3$ $\oplus$ the second $[-2]$ $\oplus$ the second $[-4]$, respectively.
    \item 
    \label{item-the-12-skein-relation}
    The following diagram can be completed to an exact triangle in $D(\C_{12}\text{-}\C_{12})$:  
        \begin{equation}
            \label{eqn-skein-relation-no-crossing}
            F_{3}^{12} R_{12}^{3}
            \rightarrow
            R_{111}^{12} F_{21}^{111} R_{111}^{21} F_{12}^{111}
            \rightarrow
            \id_{12}[-2]. 
        \end{equation}

     \begin{figure}[!h] \centering 
  \adjustbox{scale=0.9,center}{   \begin{tikzpicture}
\braid [number of strands=3,
style strands={2,3}{opacity=0}, height=1.5 cm, width=1.8cm]at (5.2,-4) a_1;
\braid [number of strands=3,
style strands={2,3}{opacity=0}, height=1.5 cm, width=1.8cm]at (5.1,-4) a_1;
\braid [number of strands=5,
style strands={2,3,4,5}{opacity=0}, height=1.5 cm, width=1 cm]at (5,-4)  a_4;

\fill[white] (3,-2) circle (2pt);

\fill[black] (5.1,-4) circle (2pt);
\fill[black] (5.2,-4) circle (2pt);

  \braid [number of strands=3,
style strands={1,3}{opacity=0}, height=1.5 cm, width=1.8cm]at (5.1,-2) a_1^{-1};
  \braid [number of strands=3,
style strands={1,3}{opacity=0}, height=1.5 cm, width=1.8cm]at (5.2,-2) a_1^{-1};
\braid [number of strands=5,
style strands={2,3,4,5}{opacity=0}, height=1.5 cm, width=2 cm]at (5,-2)  a_4;

\fill[black] (5,-2) circle (2pt);
\fill[black] (6.9,-2) circle (2pt);
\fill[black] (7,-2) circle (2pt);

\fill[black] (5,-4) circle (2pt);

\fill[black] (5,-6) circle (2pt);
\fill[black] (6.9,-6) circle (2pt);
\fill[black] (7,-6) circle (2pt);

\braid [number of strands=3,
style strands={1,2}{opacity=0}, height=0.5 cm, width=0.9cm]at (10.1,-2) a_2^{-1};
\braid [number of strands=5,
style strands={2,4,5}{opacity=0}, height=0.5 cm, width=1 cm]at (10,-2)  a_4^{-1};

\braid [number of strands=3,style strands={3,1}{opacity=0}, height=0.5 cm, width=0.9cm]at (10.1,-3) a_1^{-1};
\braid [number of strands=5, style strands={2,4,5}{opacity=0}, height=0.5 cm, width=1 cm]at (10,-3)  a_4;

\braid [number of strands=3,
style strands={2,3}{opacity=0}, height=0.5 cm, width=0.9cm]at (10.1,-4) a_1;
\braid [number of strands=5,
style strands={2,4,5}{opacity=0}, height=0.5 cm, width=1 cm]at (10,-4)  a_4;

\braid [number of strands=3,style strands={1,3}{opacity=0}, height=0.5 cm, width=0.9cm]at (10.1,-5) a_2;
\braid [number of strands=5, style strands={2,4,5}{opacity=0}, height=0.5 cm, width=1 cm]at (10,-5)  a_4;

\fill[black] (10,-2) circle (2pt);
\fill[black] (11.9,-2) circle (2pt);
\fill[black] (12,-2) circle (2pt);

\fill[black] (10,-3) circle (2pt);
\fill[black] (11,-3) circle (2pt);
\fill[black] (12,-3) circle (2pt);

\fill[black] (10,-4) circle (2pt);
\fill[black] (10.1,-4) circle (2pt);
\fill[black] (12,-4) circle (2pt);

\fill[black] (10,-5) circle (2pt);
\fill[black] (11,-5) circle (2pt);
\fill[black] (12,-5) circle (2pt);

\fill[black] (10,-6) circle (2pt);
\fill[black] (11.9,-6) circle (2pt);
\fill[black] (12,-6) circle (2pt);

\draw  node at (13,-4) [right]   {$\longrightarrow $};
\draw  node at (8,-4) [right]   {$\longrightarrow $};

\braid [number of strands=5,
style strands={2,4,5}{opacity=0}, height=3.5 cm, width=1cm]at (15,-2) a_4;
\braid [number of strands=5,
style strands={1,2,4,5}{opacity=0}, height=3.5 cm, width=1cm]at (14.9,-2) a_4;

\fill[black] (15,-2) circle (2pt);
\fill[black] (16.9,-2) circle (2pt);
\fill[black] (17,-2) circle (2pt);

\fill[black] (15,-6) circle (2pt);
\fill[black] (16.9,-6) circle (2pt);
\fill[black] (17,-6) circle (2pt);

\draw  node at (17.5,-4) [right]   {$[-2] $};

\end{tikzpicture}}

     \end{figure}

        Here the first map is the composition
        \begin{equation}
            \label{eqn-n-3-skein-first-map}
            \begin{tikzcd}
                F_{3}^{12} R^{3}_{12}
                \ar{d}{\action F_{3}^{12} R_{12}^{3} \action}
                \\
                R_{111}^{12} F_{12}^{111} F_{3}^{12} R_{12}^{3} R_{111}^{12} F_{12}^{111}
                \ar{d}{\text{multifork}}[']{\sim}
                \\
                R_{111}^{12} F_{21}^{111} F_{3}^{21} R_{21}^{3} R_{111}^{21} F_{12}^{111}
                \ar{d}{R_{111}^{12} F_{21}^{111} \trace R_{111}^{21} F_{12}^{111}}
                \\
                R_{111}^{12} F_{21}^{111} R_{111}^{21} F_{12}^{111},  
            \end{tikzcd}
        \end{equation}
        
\begin{figure}[!h] \centering 
  \adjustbox{scale=0.9,center}{ \begin{tikzpicture}

\fill[white] (-2,-2) circle (2pt);

\braid [number of strands=3,
style strands={2,3}{opacity=0}, height=1.5 cm, width=1.8cm]at (0.2,-4) a_1;
\braid [number of strands=3,
style strands={2,3}{opacity=0}, height=1.5 cm, width=1.8cm]at (0.1,-4) a_1;
\braid [number of strands=5,
style strands={2,3,4,5}{opacity=0}, height=1.5 cm, width=1 cm]at (0,-4)  a_4;

\fill[black] (0.1,-4) circle (2pt);
\fill[black] (0.2,-4) circle (2pt);

  \braid [number of strands=3,
style strands={1,3}{opacity=0}, height=1.5 cm, width=1.8cm]at (0.1,-2) a_1^{-1};
  \braid [number of strands=3,
style strands={1,3}{opacity=0}, height=1.5 cm, width=1.8cm]at (0.2,-2) a_1^{-1};
\braid [number of strands=5,
style strands={2,3,4,5}{opacity=0}, height=1.5 cm, width=2 cm]at (0,-2)  a_4;

\fill[black] (0,-2) circle (2pt);
\fill[black] (1.9,-2) circle (2pt);
\fill[black] (2,-2) circle (2pt);

\fill[black] (0,-4) circle (2pt);
\fill[black] (0.1,-4) circle (2pt);
\fill[black] (0.2,-4) circle (2pt);

\fill[black] (0,-6) circle (2pt);
\fill[black] (1.9,-6) circle (2pt);
\fill[black] (2,-6) circle (2pt);

\braid [number of strands=3,
style strands={1,2}{opacity=0}, height=0.5 cm, width=0.9cm]at (12.1,-2) a_2^{-1};
\braid [number of strands=5,
style strands={2,4,5}{opacity=0}, height=0.5 cm, width=1 cm]at (12,-2)  a_4^{-1};

\braid [number of strands=3,style strands={3,1}{opacity=0}, height=0.5 cm, width=0.9cm]at (12.1,-3) a_1^{-1};
\braid [number of strands=5, style strands={2,4,5}{opacity=0}, height=0.5 cm, width=1 cm]at (12,-3)  a_4;

\braid [number of strands=3,
style strands={2,3}{opacity=0}, height=0.5 cm, width=0.9cm]at (12.1,-4) a_1;
\braid [number of strands=5,
style strands={2,4,5}{opacity=0}, height=0.5 cm, width=1 cm]at (12,-4)  a_4;

\braid [number of strands=3,style strands={1,3}{opacity=0}, height=0.5 cm, width=0.9cm]at (12.1,-5) a_2;
\braid [number of strands=5, style strands={2,4,5}{opacity=0}, height=0.5 cm, width=1 cm]at (12,-5)  a_4;

\fill[black] (12,-2) circle (2pt);
\fill[black] (13.9,-2) circle (2pt);
\fill[black] (14,-2) circle (2pt);

\fill[black] (12,-3) circle (2pt);
\fill[black] (13,-3) circle (2pt);
\fill[black] (14,-3) circle (2pt);

\fill[black] (12,-4) circle (2pt);
\fill[black] (12.1,-4) circle (2pt);
\fill[black] (14,-4) circle (2pt);

\fill[black] (12,-5) circle (2pt);
\fill[black] (13,-5) circle (2pt);
\fill[black] (14,-5) circle (2pt);

\fill[black] (12,-6) circle (2pt);
\fill[black] (13.9,-6) circle (2pt);
\fill[black] (14,-6) circle (2pt);

\draw  node at (6.5,-4) [right]   {$\longrightarrow $};
\draw  node at (10.5,-4) [right]   {$\longrightarrow $};
\draw  node at (2.5,-4) [right]   {$\longrightarrow $};

\braid [number of strands=3,
style strands={1,3}{opacity=0}, height=0.3 cm, width=1cm]at (3.9,-5.2) a_2;
\braid [number of strands=5,
style strands={2,4,5}{opacity=0}, height=0.3 cm, width=1 cm]at (4,-5.2)  a_4;

\braid [number of strands=3,
style strands={1,2}{opacity=0}, height=0.33 cm, width=1cm]at (3.9,-4.6) a_2^{-1};
\braid [number of strands=5,
style strands={2,4,5}{opacity=0}, height=0.33 cm, width=1 cm]at (4,-4.6)  a_4^{-1};

\braid [number of strands=3,
style strands={2,3}{opacity=0}, height=0.33 cm, width=1.8cm]at (4.1,-4) a_1;
\braid [number of strands=3,
style strands={2,3}{opacity=0}, height=0.33 cm, width=1.8cm]at (4.2,-4) a_1;
\braid [number of strands=5,
style strands={2,3,4,5}{opacity=0}, height=0.33 cm, width=1 cm]at (4,-4)  a_4;

\braid [number of strands=3,
style strands={1,2}{opacity=0}, height=0.33 cm, width=1cm]at (3.9,-2) a_2^{-1};
\braid [number of strands=5,
style strands={2,4,5}{opacity=0}, height=0.33 cm, width=1 cm]at (4,-2)  a_4^{-1};

\braid [number of strands=3,
style strands={1,3}{opacity=0}, height=0.33 cm, width=1cm]at (3.9,-2.6) a_2;
\braid [number of strands=5,
style strands={2,4,5}{opacity=0}, height=0.33 cm, width=1 cm]at (4,-2.6)  a_4;

 \braid [number of strands=3,
style strands={1,3}{opacity=0}, height=0.33 cm, width=1.8cm]at (4.1,-3.2) a_1^{-1};
\braid [number of strands=3,
style strands={1,3}{opacity=0}, height=0.33 cm, width=1.8cm]at (4.2,-3.2) a_1^{-1};
\braid [number of strands=5,
style strands={2,3,4,5}{opacity=0}, height=0.33 cm, width=2 cm]at (4,-3.2)  a_4;

\fill[black] (4,-2) circle (2pt);
\fill[black] (5.9,-2) circle (2pt);
\fill[black] (6,-2) circle (2pt);

\fill[black] (4,-2.7) circle (2pt);
\fill[black] (4.9,-2.7) circle (2pt);
\fill[black] (6,-2.7) circle (2pt);

\fill[black] (4,-3.3) circle (2pt);
\fill[black] (5.9,-3.3) circle (2pt);
\fill[black] (6,-3.3) circle (2pt);

\fill[black] (4,-4.7) circle (2pt);
\fill[black] (5.9,-4.7) circle (2pt);
\fill[black] (6,-4.7) circle (2pt);

\fill[black] (4,-5.3) circle (2pt);
\fill[black] (4.9,-5.3) circle (2pt);
\fill[black] (6,-5.3) circle (2pt);

\fill[black] (4,-4) circle (2pt);
\fill[black] (4.1,-4) circle (2pt);
\fill[black] (4.2,-4) circle (2pt);

\fill[black] (4,-6) circle (2pt);
\fill[black] (5.9,-6) circle (2pt);
\fill[black] (6,-6) circle (2pt);

\braid [number of strands=3,
style strands={1,3}{opacity=0}, height=0.3 cm, width=1cm]at (7.9,-5.2) a_2;
\braid [number of strands=5,
style strands={2,4,5}{opacity=0}, height=0.3 cm, width=1 cm]at (8,-5.2)  a_4;

\braid [number of strands=3,
style strands={1,2}{opacity=0}, height=0.33 cm, width=1cm]at (7.9,-2) a_2^{-1};
\braid [number of strands=5,
style strands={2,4,5}{opacity=0}, height=0.33 cm, width=1 cm]at (8,-2)  a_4^{-1};

\braid [number of strands=3,
style strands={2,3}{opacity=0}, height=0.33 cm, width=0.8cm]at (8.1,-4.6) a_1;
\braid [number of strands=5,
style strands={2,4,5}{opacity=0}, height=0.33 cm, width=1 cm]at (8,-4.6)  a_4;

\braid [number of strands=3,
style strands={2,3}{opacity=0}, height=0.3 cm, width=1.8cm]at (8.2,-4) a_1;

\braid [number of strands=5,
style strands={2,3,4,5}{opacity=0}, height=0.33 cm, width=1 cm]at (8,-4)  a_4;
\braid [number of strands=5,
style strands={2,3,4,5}{opacity=0}, height=0.33 cm, width=1 cm]at (8.1,-4)  a_4;

  \braid [number of strands=3,
style strands={1,3}{opacity=0}, height=0.33 cm, width=0.8cm]at (8.1,-2.6) a_1^{-1};
\braid [number of strands=5,
style strands={2,4,5}{opacity=0}, height=0.33 cm, width=1 cm]at (8,-2.6)  a_4;

 \braid [number of strands=3,
style strands={1,3}{opacity=0}, height=0.33 cm, width=1.8cm]at (8.2,-3.2) a_1^{-1};

\braid [number of strands=5,
style strands={2,3,4,5}{opacity=0}, height=0.33 cm, width=1 cm]at (8,-3.2)  a_4;
\braid [number of strands=5,
style strands={2,3,4,5}{opacity=0}, height=0.33 cm, width=1 cm]at (8.1,-3.2)  a_4;

\fill[black] (8,-2) circle (2pt);
\fill[black] (9.9,-2) circle (2pt);
\fill[black] (10,-2) circle (2pt);

\fill[black] (8,-2.7) circle (2pt);
\fill[black] (8.9,-2.7) circle (2pt);
\fill[black] (10,-2.7) circle (2pt);

\fill[black] (8,-3.3) circle (2pt);
\fill[black] (8.1,-3.3) circle (2pt);
\fill[black] (10,-3.3) circle (2pt);

\fill[black] (8,-4) circle (2pt);
\fill[black] (8.1,-4) circle (2pt);
\fill[black] (8.2,-4) circle (2pt);

\fill[black] (8,-5.3) circle (2pt);
\fill[black] (8.9,-5.3) circle (2pt);
\fill[black] (10,-5.3) circle (2pt);

\fill[black] (8,-4.7) circle (2pt);
\fill[black] (8.1,-4.7) circle (2pt);
\fill[black] (10,-4.7) circle (2pt);

\fill[black] (8,-6) circle (2pt);
\fill[black] (9.9,-6) circle (2pt);
\fill[black] (10,-6) circle (2pt);

\end{tikzpicture} }
\end{figure}

        and the second map is the composition
        \begin{equation}
            \label{eqn-n-3-skein-second-map}
            R_{111}^{12} F_{21}^{111} R_{111}^{21} F_{12}^{111}
            \simeq
            L_{111}^{12} F_{21}^{111} R_{111}^{21} F_{12}^{111}[-2]
            \xrightarrow{\trace[-2]}
            \id_{12}[-2].
        \end{equation}
    \item 
    \label{item-the-21-skein-relation}
    The following diagram can be completed to an exact triangle in $D(\C_{21}\text{-}\C_{21})$:  
        \begin{equation}
            \label{eqn-skein-relation21-no-crossing}
            F_{3}^{21} R_{21}^{3}
            \rightarrow
            R_{111}^{21} F_{12}^{111} R_{111}^{12} F_{21}^{111}
            \rightarrow
            \id_{21}[-2]. 
        \end{equation}
        Its two maps are defined analogously to \eqref{eqn-n-3-skein-first-map} and \eqref{eqn-n-3-skein-second-map}:
        \begin{equation}
            \label{eqn-n-3-skein-first-map-21}
            \begin{tikzcd}
                F_{3}^{21} R^{3}_{21}
                \ar{d}{\action F_{3}^{21} R_{21}^{3} \action}
                \\
                R_{111}^{21} F_{21}^{111} F_{3}^{21} R_{21}^{3} R_{111}^{21} F_{21}^{111}
                \ar{d}{\text{multifork}}[']{\sim}
                \\
                R_{111}^{21} F_{12}^{111} F_{3}^{12} R_{12}^{3} R_{111}^{12} F_{21}^{111}
                \ar{d}{R_{111}^{21} F_{12}^{111} \trace R_{111}^{12} F_{21}^{111}}
                \\
                R_{111}^{21} F_{12}^{111} R_{111}^{12} F_{21}^{111},  
            \end{tikzcd}
        \end{equation}
        and 
        \begin{equation}
            \label{eqn-n-3-skein-second-map-21}
            R_{111}^{21} F_{12}^{111} R_{111}^{12} F_{21}^{111}
            \simeq
            L_{111}^{21} F_{12}^{111} R_{111}^{12} F_{21}^{111}[-2]
            \xrightarrow{\trace[-2]}
            \id_{21}[-2].
        \end{equation}    
    
    \item \label{item-valency-2-merge-condition} 
    For $(1,1)$-merges $G_{\bar{i}}^{\bar{j}}$ we have
        \begin{equation}
            L_{\bar{i}}^{\bar{j}}[-1] \simeq G_{\bar{i}}^{\bar{j}} \simeq R_{\bar{i}}^{\bar{j}}[1].
        \end{equation}
    \item \label{item-valency-3-merge-condition}
    For $(1,2)$- and $(2,1)$-merges $G_{\bar{i}}^{\bar{j}}$ we have 
        \begin{equation}
            L_{\bar{i}}^{\bar{j}}[-2] \simeq G_{\bar{i}}^{\bar{j}} \simeq R_{\bar{i}}^{\bar{j}}[2].
        \end{equation}
    \item \label{item-1-1-crossings-condition} 
    $(1,1)$-crossings $T_{\bar{i}}^{\bar{j}}$, $D_{\bar{i}}^{\bar{j}}$ are isomorphic to the spherical twists of $(1,1)$-forks: 
        \begin{align}
            \brTaaaCa \simeq \cone\left( F_{21}^{111} G_{111}^{21}[-1] \xrightarrow{\trace} \id_{111} \right), 
            \\
            \brTaaaCb \simeq \cone\left( F_{12}^{111}G_{111}^{12}[-1] \xrightarrow{\trace} \id_{111} \right), 
            \\
            \brDaaaCa \simeq \cone\left( \id_{111}[-1] \xrightarrow{\action} F_{21}^{111} G_{111}^{21} \right), 
            \\
            \brDaaaCb \simeq \cone\left( \id_{111}[-1] \xrightarrow{\action} F_{12}^{111} G_{111}^{12} \right), 
        \end{align}
        
\begin{figure}[!h] \centering 
\adjustbox{scale=0.9,center}{
\begin{tikzpicture}

\braid [number of strands=3,style strands={3,1}{opacity=0}, height=0.5 cm, width=0.9cm]at (0.1,-2) a_1^{-1};
\braid [number of strands=5, style strands={2,4,5}{opacity=0}, height=0.5 cm, width=1 cm]at (0,-2)  a_4;

\fill[white] (-8,-4) circle (2pt);

\braid [number of strands=3,
style strands={2,3}{opacity=0}, height=0.5 cm, width=0.9cm]at (0.1,-3) a_1;
\braid [number of strands=5,
style strands={2,4,5}{opacity=0}, height=0.5 cm, width=1 cm]at (0,-3)  a_4;

\fill[black] (0,-2) circle (2pt);
\fill[black] (1,-2) circle (2pt);
\fill[black] (2,-2) circle (2pt);

\fill[black] (0,-3) circle (2pt);
\fill[black] (0.1,-3) circle (2pt);
\fill[black] (2,-3) circle (2pt);

\fill[black] (0,-4) circle (2pt);
\fill[black] (1,-4) circle (2pt);
\fill[black] (2,-4) circle (2pt);

\draw  node at (4,-3) [right]   {$\longrightarrow$};

\braid [number of strands=5,
style strands={4,5}{opacity=0}, height=1.5 cm, width=1cm]at (6,-2) a_4;

\fill[black] (8,-2) circle (2pt);
\fill[black] (7,-2) circle (2pt);
\fill[black] (6,-2) circle (2pt);

\fill[black] (8,-4) circle (2pt);
\fill[black] (7,-4) circle (2pt);
\fill[black] (6,-4) circle (2pt);

\draw  node at (-3,-3) [right]   {$ = $};

\draw  node at (-2,-3) [right]   {$ \mathrm{Cone }$};
\draw  node at (-1,-3) [right]   {$ \mathrm{( }$};
\draw  node at (9,-3) [right]   {$ \mathrm{) }$};
\draw  node at (2.5,-3) [right]   {$ \mathrm{[-1] }$};

\braid [number of strands=3, height=1.5 cm, width=1cm]at (-6,-2) a_1^{-1};

\fill[black] (-6,-2) circle (2pt);
\fill[black] (-6,-2) circle (2pt);
\fill[black] (-5,-2) circle (2pt);
\fill[black] (-4,-2) circle (2pt);
\fill[black] (-6,-4) circle (2pt);
\fill[black] (-5,-4) circle (2pt);
\fill[black] (-4,-4) circle (2pt);

\end{tikzpicture} }
\end{figure}

\begin{figure}[!h] \centering 
\adjustbox{scale=0.9,center}{
\begin{tikzpicture}

\braid [number of strands=3,style strands={3,1}{opacity=0}, height=0.5 cm, width=0.9cm]at (6.1,-2) a_1^{-1};
\braid [number of strands=5, style strands={2,4,5}{opacity=0}, height=0.5 cm, width=1 cm]at (6,-2)  a_4;

\fill[white] (-8,-4) circle (2pt);

\braid [number of strands=3,
style strands={2,3}{opacity=0}, height=0.5 cm, width=0.9cm]at (6.1,-3) a_1;
\braid [number of strands=5,
style strands={2,4,5}{opacity=0}, height=0.5 cm, width=1 cm]at (6,-3)  a_4;

\fill[black] (6,-2) circle (2pt);
\fill[black] (7,-2) circle (2pt);
\fill[black] (8,-2) circle (2pt);

\fill[black] (6,-3) circle (2pt);
\fill[black] (6.1,-3) circle (2pt);
\fill[black] (8,-3) circle (2pt);

\fill[black] (6,-4) circle (2pt);
\fill[black] (7,-4) circle (2pt);
\fill[black] (8,-4) circle (2pt);

\draw  node at (4,-3) [right]   {$\longrightarrow$};

\braid [number of strands=5,
style strands={4,5}{opacity=0}, height=1.5 cm, width=1cm]at (0,-2) a_4;

\fill[black] (2,-2) circle (2pt);
\fill[black] (1,-2) circle (2pt);
\fill[black] (0,-2) circle (2pt);

\fill[black] (2,-4) circle (2pt);
\fill[black] (1,-4) circle (2pt);
\fill[black] (0,-4) circle (2pt);

\draw  node at (-3,-3) [right]   {$ = $};

\draw  node at (-2,-3) [right]   {$ \mathrm{Cone }$};
\draw  node at (-1,-3) [right]   {$ \mathrm{( }$};
\draw  node at (9,-3) [right]   {$ \mathrm{) }$};
\draw  node at (2.5,-3) [right]   {$ \mathrm{[-1] }$};

\braid [number of strands=3, height=1.5 cm, width=1cm]at (-6,-2) a_1;

\fill[black] (-6,-2) circle (2pt);
\fill[black] (-6,-2) circle (2pt);
\fill[black] (-5,-2) circle (2pt);
\fill[black] (-4,-2) circle (2pt);
\fill[black] (-6,-4) circle (2pt);
\fill[black] (-5,-4) circle (2pt);
\fill[black] (-4,-4) circle (2pt);

\end{tikzpicture} }
\end{figure}

    \item
    \label{item-1-2-and-2-1-crossings-condition}
    For $(1,2)$- and $(2,1)$-crossings $T_{\bar{i}}^{\bar{j}}$, $D_{\bar{i}}^{\bar{j}}$ we have:
        \begin{align}
            T_{12}^{21} \simeq \cone\left( F_{3}^{21} G_{12}^{3}[-1] \xrightarrow{\lambda} G_{111}^{21} F_{12}^{111} \right), 
            \\
            T_{21}^{12} \simeq \cone\left( F_{3}^{12} G_{21}^{3}[-1] \xrightarrow{\mu} G_{111}^{12} F_{21}^{111} \right), 
            \\
            D_{12}^{21} \simeq \cone\left( G_{111}^{21} F_{12}^{111}[-1] \xrightarrow{\lambda'} F_{3}^{21} G_{12}^{3}  \right), 
            \\
            D_{21}^{12} \simeq \cone\left( G_{111}^{12} F_{21}^{111}[-1] \xrightarrow{\mu'} F_{3}^{12} G_{21}^{3}  \right), 
        \end{align}

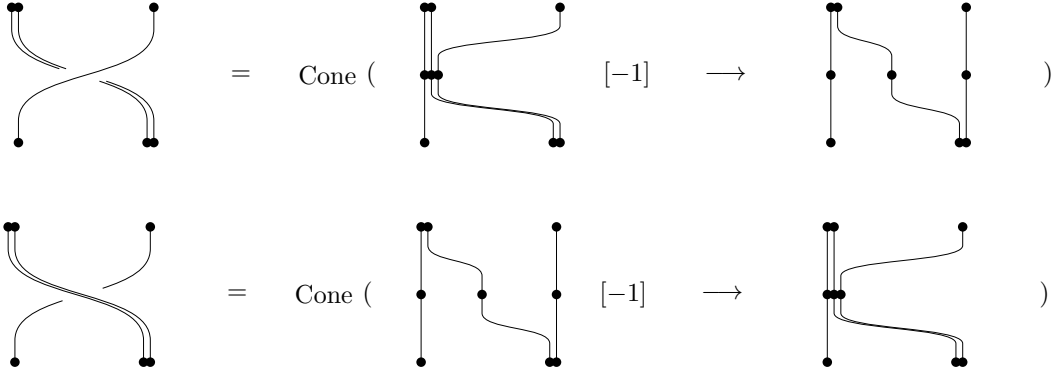
\begin{figure}[!h] \centering 
\adjustbox{scale=0.9,center}{ \begin{tikzpicture}

\braid [number of strands=3,
style strands={1,3}{opacity=0}, height=0.5 cm, width=1cm]at (5.9,-3) a_2;
\braid [number of strands=5,
style strands={2,4,5}{opacity=0}, height=0.5 cm, width=1 cm]at (6,-3)  a_4;

\braid [number of strands=3,
style strands={2,3}{opacity=0}, height=0.5 cm, width=0.8cm]at (6.1,-2) a_1;
\braid [number of strands=5,
style strands={2,4,5}{opacity=0}, height=0.5 cm, width=1 cm]at (6,-2)  a_4;

\fill[black] (6,-2) circle (2pt);
\fill[black] (6.1,-2) circle (2pt);
\fill[black] (8,-2) circle (2pt);

\fill[black] (6,-3) circle (2pt);
\fill[black] (6.9,-3) circle (2pt);
\fill[black] (8,-3) circle (2pt);

\fill[black] (6,-4) circle (2pt);
\fill[black] (7.9,-4) circle (2pt);
\fill[black] (8,-4) circle (2pt);

\draw  node at (4,-3) [right]   {$\longrightarrow$};

\braid [number of strands=3,
style strands={2,3}{opacity=0}, height=0.5 cm, width=1.8cm]at (0.1,-3) a_1;
\braid [number of strands=3,
style strands={2,3}{opacity=0}, height=0.5 cm, width=1.8cm]at (0.2,-3) a_1;
\braid [number of strands=5,
style strands={2,3,4,5}{opacity=0}, height=0.5 cm, width=1 cm]at (0,-3)  a_4;

 \braid [number of strands=3,
style strands={1,2}{opacity=0}, height=0.5 cm, width=1.8cm]at (-3.5,-2) a_1;
\braid [number of strands=3,
style strands={1,3}{opacity=0}, height=0.5 cm, width=1.8cm]at (0.2,-2) a_1^{-1};
\braid [number of strands=5,
style strands={2,3,4,5}{opacity=0}, height=0.5 cm, width=2 cm]at (0,-2)  a_4;

\fill[black] (2,-2) circle (2pt);
\fill[black] (0.1,-2) circle (2pt);
\fill[black] (0,-2) circle (2pt);

\fill[black] (0.2,-3) circle (2pt);
\fill[black] (0.1,-3) circle (2pt);
\fill[black] (0,-3) circle (2pt);

\fill[black] (2,-4) circle (2pt);
\fill[black] (1.9,-4) circle (2pt);
\fill[black] (0,-4) circle (2pt);

\draw  node at (-3,-3) [right]   {$ = $};

\draw  node at (-2,-3) [right]   {$ \mathrm{Cone }$};
\draw  node at (-1,-3) [right]   {$ \mathrm{( }$};
\draw  node at (9,-3) [right]   {$ \mathrm{) }$};
\draw  node at (2.5,-3) [right]   {$ \mathrm{[-1] }$};

\braid [number of strands=2,
style strands={3}{opacity=0}, height=1.5 cm, width=2 cm]at (-6,-2)  a_1^{-1};
\braid [number of strands=2,
style strands={2}{opacity=0}, height=1.5 cm, width=2 cm]at (-6.1,-2)  a_1^{-1};
\fill[black] (-6.1,-2) circle (2pt);
\fill[black] (-6,-2) circle (2pt);
\fill[black] (-4,-2) circle (2pt);
\fill[black] (-6,-4) circle (2pt);
\fill[black] (-4.1,-4) circle (2pt);
\fill[black] (-4,-4) circle (2pt);

\end{tikzpicture} }
\end{figure}

\begin{figure}[!h] \centering 
\adjustbox{scale=0.9,center}{ \begin{tikzpicture}
\fill[white] (-10,-2) circle (2pt);

\braid [number of strands=3,
style strands={1,3}{opacity=0}, height=0.5 cm, width=1cm]at (-0.1,-3) a_2;
\braid [number of strands=5,
style strands={2,4,5}{opacity=0}, height=0.5 cm, width=1 cm]at (0,-3)  a_4;

\braid [number of strands=3,
style strands={2,3}{opacity=0}, height=0.5 cm, width=0.8cm]at (0.1,-2) a_1;
\braid [number of strands=5,
style strands={2,4,5}{opacity=0}, height=0.5 cm, width=1 cm]at (0,-2)  a_4;

\fill[black] (0,-2) circle (2pt);
\fill[black] (0.1,-2) circle (2pt);
\fill[black] (2,-2) circle (2pt);

\fill[black] (0,-3) circle (2pt);
\fill[black] (0.9,-3) circle (2pt);
\fill[black] (2,-3) circle (2pt);

\fill[black] (0,-4) circle (2pt);
\fill[black] (1.9,-4) circle (2pt);
\fill[black] (2,-4) circle (2pt);

\draw  node at (4,-3) [right]   {$\longrightarrow$};

\braid [number of strands=3,
style strands={2,3}{opacity=0}, height=0.5 cm, width=1.8cm]at (6.1,-3) a_1;
\braid [number of strands=3,
style strands={2,3}{opacity=0}, height=0.5 cm, width=1.8cm]at (6.2,-3) a_1;
\braid [number of strands=5,
style strands={2,3,4,5}{opacity=0}, height=0.5 cm, width=1 cm]at (6,-3)  a_4;

 \braid [number of strands=3,
style strands={1,2}{opacity=0}, height=0.5 cm, width=1.8cm]at (2.5,-2) a_1;
\braid [number of strands=3,
style strands={1,3}{opacity=0}, height=0.5 cm, width=1.8cm]at (6.2,-2) a_1^{-1};
\braid [number of strands=5,
style strands={2,3,4,5}{opacity=0}, height=0.5 cm, width=2 cm]at (6,-2)  a_4;

\fill[black] (8,-2) circle (2pt);
\fill[black] (6.1,-2) circle (2pt);
\fill[black] (6,-2) circle (2pt);

\fill[black] (6.2,-3) circle (2pt);
\fill[black] (6.1,-3) circle (2pt);
\fill[black] (6,-3) circle (2pt);

\fill[black] (8,-4) circle (2pt);
\fill[black] (7.9,-4) circle (2pt);
\fill[black] (6,-4) circle (2pt);

\draw  node at (-3,-3) [right]   {$ = $};

\draw  node at (-2,-3) [right]   {$ \mathrm{Cone }$};
\draw  node at (-1,-3) [right]   {$ \mathrm{( }$};
\draw  node at (9,-3) [right]   {$ \mathrm{) }$};
\draw  node at (2.5,-3) [right]   {$ \mathrm{[-1] }$};

\braid [number of strands=2,
style strands={3}{opacity=0}, height=1.5 cm, width=2 cm]at (-6,-2)  a_1;
\braid [number of strands=2,
style strands={2}{opacity=0}, height=1.5 cm, width=2 cm]at (-6.1,-2)  a_1;
\fill[black] (-6.1,-2) circle (2pt);
\fill[black] (-6,-2) circle (2pt);
\fill[black] (-4,-2) circle (2pt);
\fill[black] (-6,-4) circle (2pt);
\fill[black] (-4.1,-4) circle (2pt);
\fill[black] (-4,-4) circle (2pt);

\end{tikzpicture} }
\end{figure}
        
        where $\lambda$ is defined by either of the two equal compositions 
        \begin{equation}
            \begin{tikzcd}
                F_{3}^{21} R_{12}^{3}
                \ar{d}{F_{3}^{21} R_{12}^{3}\action}
                \\
                F_{3}^{21} R_{12}^{3} R_{111}^{12} F_{12}^{111}
                \ar{d}{\text{multifork}}[']{\sim}
                \\
                F_{3}^{21} R_{21}^{3} R_{111}^{21} F_{12}^{111}
                \ar{d}{\trace R_{111}^{21} F_{12}^{111}}
                \\
                R_{111}^{21} F_{12}^{111}
            \end{tikzcd}
            \quad 
            \begin{tikzcd}
                F_{3}^{21} R_{12}^{3}
                \ar{d}{\action F_{3}^{21} R_{12}^{3}}
                \\
                R_{111}^{21} F_{21}^{111} F_{3}^{21} R_{12}^{3} 
                \ar{d}{\text{multifork}}[']{\sim}
                \\
                R_{111}^{21} F_{12}^{111} F_{3}^{12} R_{12}^{3} 
                \ar{d}{R_{111}^{21} F_{12}^{111} \trace}
                \\
                R_{111}^{21} F_{12}^{111}
            \end{tikzcd}
            ,
        \end{equation}
        $\lambda'$ is defined by either of the two equal compositions
        \begin{equation}
            \begin{tikzcd}
                L_{111}^{21} F_{12}^{111}
                \ar{d}{L_{111}^{21} F_{12}^{111} \action}
                \\
                L_{111}^{21} F_{12}^{111} F_{3}^{12} L_{12}^{3} 
                \ar{d}{\text{multifork}}[']{\sim}
                \\
                L_{111}^{21} F_{21}^{111} F_{3}^{21} L_{12}^{3}
                \ar{d}{\trace F_{3}^{21} L_{12}^{3}}
                \\
                F_{3}^{21} L_{12}^{3}
            \end{tikzcd}
            \quad 
            \begin{tikzcd}
                L_{111}^{21} F_{12}^{111}
                \ar{d}{ \action L_{111}^{21} F_{12}^{111}}
                \\
                F_{3}^{21} L_{21}^{3} L_{111}^{21} F_{12}^{111}  
                \ar{d}{\text{multifork}}[']{\sim}
                \\
                F_{3}^{21} L_{12}^{3} L_{111}^{12} F_{12}^{111} 
                \ar{d}{F_{3}^{21} L_{12}^{3}\trace}
                \\
                F_{3}^{21} L_{12}^{3}
            \end{tikzcd}
            ,
        \end{equation}
        and $\mu$ and $\mu'$ are defined similarly.
 
\end{enumerate}
\end{definition}

Such split-$\mathbb{P}^2$ representations of $\gbrcat_3$ are a subclass of skein-triangulated representations of $\gbrcat_3$, the notion defined and studied in \cite{AL4}.

\subsection{A sufficient condition for a split-$\mathbb{P}^2$ categorical action of $\gbrcat_3$}
The following is our main result in this section. It shows that if a collection of fork functors satisfies a certain subset of conditions in Definitions \ref{defn-split-p^2-representation-of-gbr-3}, we can generate a split-$\mathbb{P}^2$ representation of $\gbrcat_3$ from them:
\begin{theorem}\label{Theorem1}

Let $\C_{3}$, $\C_{12}$, $\C_{21}$, and $\C_{111}$ be enhanced triangulated categories. Let
\begin{align}
    F_{3}^{12}&\colon \C_{3} \rightarrow \C_{12} \\
    F_{3}^{21}&\colon \C_{3} \rightarrow \C_{21} \\
    F_{12}^{111}&\colon C_{12} \rightarrow \C_{111} \\
    F_{21}^{111}&\colon C_{21} \rightarrow C_{111}
\end{align}
be enhanced exact functors which satisfy the conditions 
\eqref{item-existence-of-both-adjoints}-\eqref{item-the-21-skein-relation}
of Definition \ref{defn-split-p^2-representation-of-gbr-3}. 

Then setting $F_{\bar{i}}^{\bar{j}}$ to be the fork functors, and defining merge and crossing functors  $G_{\bar{i}}^{\bar{j}}$, $T_{\bar{i}}^{\bar{j}}$, $D_{\bar{i}}^{\bar{j}}$ by the formulas in the conditions 
\eqref{item-valency-2-merge-condition}-\eqref{item-1-2-and-2-1-crossings-condition} of Definition \ref{defn-split-p^2-representation-of-gbr-3}, we obtain a split-$\mathbb{P}^2$ categorical action of $\gbrcat_3$. 

Moreover, this action also satisfies the "Flop + Flop = Twist" condition:

\begin{enumerate}
    \item $T_{21}^{12} T_{12}^{21}$ is isomorphic in $D(\C_{12}\text{-}\C_{12})$ to the $\mathbb{P}$-twist of $F_{3}^{12}$,
    \item $T_{12}^{21} T_{21}^{12}$ is isomorphic in $D(\C_{21}\text{-}\C_{21})$ to the $\mathbb{P}$-twist of $F_{3}^{21}$.
\end{enumerate}    

\end{theorem}

\begin{proof}

It suffices to prove that the four basic relations between the generators of $\gbrcat_3$ listed in \eqref{list-generalised-braid-relations-for-gbr3} hold under the assumptions of the theorem. The proofs for the relations obtained from these four by vertical, horizontal and blackboard reflections are identical. 
\\
\\
\textbf{$(\brTaaaCa,\brDaaaCa)$ and $(\brTaaaCb,\brDaaaCb)$ are pairs of mutually inverse equivalences.}

This follows from $F_{21}^{111}$ and $F_{12}^{111}$ being spherical functors, see \cite{AL2}, Theorem 5.1. 
\\
\\
\textbf{The pitchfork relation}

We need to show the existence of an isomorphism
\begin{equation*}
    F_{12}^{111} T_{21}^{12} \simeq T_{\crossn{1}\crossn{1}1}^{111} T_{1\crossn{1}\crossn{1}}^{111} F_{21}^{111}.
\end{equation*}

By definition we have
        \begin{align*}
            \brTaaaCa = \cone\left( F_{21}^{111} G_{111}^{21}[-1] \xrightarrow{\trace} \id_{111} \right), 
            \\
            \brTaaaCb = \cone\left( F_{12}^{111}G_{111}^{12}[-1] \xrightarrow{\trace} \id_{111} \right), 
        \end{align*}
        therefore $ T_{\crossn{1}\crossn{1}1}^{111} T_{1\crossn{1}\crossn{1}}^{111} F_{21}^{111}$ is isomorphic in $D(\mathcal{C}_{111}-\mathcal{C}_{111})$ to the convolution of the twisted complex
\begin{scriptsize}
\begin{equation*}
    F_{21}^{111}R_{111}^{21}F_{12}^{111}R_{111}^{12}F_{21}^{111} \xrightarrow{\left( \begin{smallmatrix} tr F_{12}^{111} R_{111}^{12} F_{21}^{111} \\ -F_{21}^{111}R_{111}^{21} tr F_{21}^{111}  \end{smallmatrix} \right)} F_{12}^{111}R_{111}^{12}F_{21}^{111} \oplus F_{21}^{111}R_{111}^{21}F_{21}^{111} \xrightarrow{\left( tr F_{21} \quad  tr F_{21}^{111} \right)}\underset{\degzero} {F_{21}^{111}}
\end{equation*}
\end{scriptsize}

By Lemma 5.10 of \cite{AL2},  between 
\[ \underset{ \degzero} {F_{21}^{111}R_{111}^{21}F_{21}^{111}} \xrightarrow{tr F_{21}^{111}} F_{21}^{111} \text{ and } F_{21}^{111} \xrightarrow{F_{21}^{111} act} \underset{ \degzero} {F_{21}^{111}R_{111}^{21}F_{21}^{111}}\] 

we have a homotopy equivalence of twisted complexes.

It follows by the Replacement Lemma  (\cite{AL3}, Lemma 2.1) that $ T_{\crossn{1}\crossn{1}1}^{111} T_{1\crossn{1}\crossn{1}}^{111} F_{21}^{111}$ is further isomorphic to the convolution of 
\begin{scriptsize}
\begin{equation*}
    F_{21}^{111}R_{111}^{21}F_{12}^{111}R_{111}^{12}F_{21}^{111} \oplus F_{21}^{111} \xrightarrow{\left( \begin{smallmatrix} tr F_{12}^{111} R_{111}^{12} F_{21}^{111} & 0 \\  -F_{21}^{111} R_{111}^{21} tr F_{21}^{111} & F_{21}^{111} act  \end{smallmatrix} \right)} F_{12}^{111} R_{111}^{12} F_{21}^{111}  \underset{ \degminusone}{\oplus} F_{21}^{111} R_{111}^{21} F_{21}^{111}.
\end{equation*}
\end{scriptsize}

Moreover, there exists a homotopy equivalence from $F_{21}^{111}F_{3}^{21}R_{21}^3$ to
\begin{scriptsize}
\begin{equation*}
    F_{21}^{111}R_{111}^{21}F_{12}^{111}R_{111}^{12}F_{21}^{111} \oplus F_{21}^{111} \xrightarrow{\left( -F_{21}^{111}R_{111}^{21} tr F_{21}^{111} \quad F_{21}^{111} act \right) } \underset{\degzero }{F_{21}^{111} R_{111}^{21} F_{21}^{111}}
\end{equation*}
\end{scriptsize}

whose $F_{21}^{111}F_{3}^{21}R_{21}^3 \to  F_{21}^{111}R_{111}^{21}F_{12}^{111}R_{111}^{12}F_{21}^{111}$ component is the map $F_{21}^{111} (\ref{eqn-n-3-skein-first-map})$.

It follows by the Replacement Lemma, that $ T_{\crossn{1}\crossn{1}1}^{111} T_{1\crossn{1}\crossn{1}}^{111} F_{21}^{111}$ is therefore isomorphic to the convolution of 

\begin{equation*}
    F_{21}^{111}F_{3}^{21}R_{21}^{3} \xrightarrow{tr F_{12}^{111}R_{111}F_{21}^{111}\circ F_{21}^{111} (\ref{eqn-n-3-skein-first-map})} \underset{\degminusone}{F_{12}^{111}R_{111}^{12}F_{21}^{111}}
\end{equation*}
and hence by the multifork isomorphism to the convolution of 

\begin{equation*}
    F_{12}^{111}F_{3}^{12}R_{21}^{3} \xrightarrow{tr F_{12}^{111}R_{111}F_{21}^{111}\circ F_{21}^{111} (\ref{eqn-n-3-skein-first-map})\circ multifork} \underset{\degminusone}{F_{12}^{111}R_{111}^{12}F_{21}^{111}}.
\end{equation*}

On the other hand, by the definition of $T_{21}^{12}$ the object $F_{12}^{111} T_{21}^{12}$ is isomorphic in $D(\mathcal{C}_{111}-\mathcal{C}_{111})$ to the convolution if the twisted complex

\begin{equation*}
    F_{12}^{111}F_{3}^{12}R_{21}^{3} \xrightarrow{F_{12}^{111} \mu } \underset{\degminusone}{F_{12}^{111}R_{111}^{12}F_{21}^{111}}.
\end{equation*}

And since 
\[ tr F_{12}^{111}R_{111}F_{21}^{111}\circ F_{21}^{111} (\ref{eqn-n-3-skein-first-map})\circ multifork = F_{12}^{111} \mu \]
we have

\begin{equation*}
    F_{12}^{111} T_{21}^{12} \simeq T_{\crossn{1}\crossn{1}1}^{111} T_{1\crossn{1}\crossn{1}}^{111} F_{21}^{111}.
\end{equation*}
\\
\\
\textbf{The braid relation}

Consider the following twisted complexes of enhanced functors: 
\begin{scriptsize}
\begin{align}
\label{eqn-O1-definition}
\underset{\degzero}{F_{21}^{111} R_{111}^{21} F_{12}^{111} R_{111}^{12}  F_{21}^{111} R_{111}^{21} \oplus F_{21}^{111} R_{111}^{21}}
& \xrightarrow{F_{21}^{111} R_{111}^{21} \trace F_{21}^{111} R_{111}^{21} \oplus F_{21}^{111} \action R_{111}^{21}}
F_{21}^{111} R_{111}^{21} F_{21}^{111} R_{111}^{21},
\\
\label{eqn-O2-definition}
\underset{\degzero}{F_{12}^{111} R_{111}^{12} F_{21}^{111} R_{111}^{21}  F_{12}^{111} R_{111}^{12} \oplus F_{12}^{111} R_{111}^{12}}
& \xrightarrow{F_{12}^{111} R_{111}^{12} \trace F_{12}^{111} R_{111}^{12} \oplus F_{12}^{111} \action R_{111}^{12}}
F_{12}^{111} R_{111}^{12} F_{12}^{111} R_{111}^{12}.
\end{align}
\end{scriptsize}

There are natural maps from both to $F_{21}^{111} R_{111}^{21} F_{12}^{111} R_{111}^{12} \oplus F_{12}^{111} R_{111}^{12}F_{21}^{111} R_{111}^{21}$ 
induced by the maps
\begin{equation}
\label{eqn-map-O1-direct-sum}
\begin{tikzcd}
F_{21}^{111} R_{111}^{21} F_{12}^{111} R_{111}^{12}  F_{21}^{111} R_{111}^{21}
\ar{d}{\trace F_{12}^{111} R_{111}^{12}  F_{21}^{111} R_{111}^{21} \oplus F_{21}^{111} R_{111}^{21} F_{12}^{111} R_{111}^{12} \trace}
\\
F_{21}^{111} R_{111}^{21} F_{12}^{111} R_{111}^{12} \oplus F_{12}^{111} R_{111}^{12}F_{21}^{111} R_{111}^{21}
\end{tikzcd}
\end{equation}
and
\begin{equation}
\label{eqn-map-O2-direct-sum}
\begin{tikzcd}
F_{12}^{111} R_{111}^{12} F_{21}^{111} R_{111}^{21}  F_{12}^{111} R_{111}^{12}
\ar{d}{ F_{12}^{111} R_{111}^{12} F_{21}^{111} R_{111}^{21} \trace \oplus \trace  F_{21}^{111} R_{111}^{21}  F_{12}^{111} R_{111}^{12}}
\\
F_{21}^{111} R_{111}^{21} F_{12}^{111} R_{111}^{12} \oplus F_{12}^{111} R_{111}^{12}F_{21}^{111} R_{111}^{21}.
\end{tikzcd}
\end{equation}
By \cite{AL2} Theorem 6.2, if there exists a $D(\C_{111}\text{-}\C_{111})$ isomorphism between the convolutions 
of the twisted complexes \eqref{eqn-O1-definition} and \eqref{eqn-O2-definition} which intertwines the maps 
induced by \eqref{eqn-map-O1-direct-sum} and \eqref{eqn-map-O2-direct-sum}, then 
the braid relation holds for $\brTaaaCa$ and $\brTaaaCb$.

Now, recall that we have the multifork isomorphism $\alpha\colon F_{21}^{111}F_{3}^{21} \xrightarrow{\sim} F_{12}^{111}F_{3}^{12}$. 
Let $\beta$ denote its inverse $F_{12}^{111}F_{3}^{12} \xrightarrow{\sim} F_{21}^{111}F_{3}^{21}$, and 
let $\beta^{R}\colon R_{21}^{3} R_{111}^{21}\xrightarrow{\sim} R_{12}^{3} R_{111}^{12}$. 

We claim that:

\begin{enumerate}
    \item[(a)] There exist homotopy equivalences from the objects $F_{21}^{111}F_3^{21} R_{21}^{3} R_{111}^{21}$
and $F_{12}^{111}F_3^{12} R_{12}^{3} R_{111}^{12}$ to the twisted complexes \eqref{eqn-O1-definition} and \eqref{eqn-O2-definition}
whose 
$$ F_{21}^{111}F_3^{21} R_{21}^{3} R_{111}^{21} \rightarrow F_{21}^{111} R_{111}^{21} F_{12}^{111} R_{111}^{12}  F_{21}^{111} R_{111}^{21}, $$
$$ F_{12}^{111}F_3^{12} R_{12}^{3} R_{111}^{12} \rightarrow F_{12}^{111} R_{111}^{12} F_{21}^{111} R_{111}^{21}  F_{12}^{111} R_{111}^{12}, $$
components are the maps 
$F_{21}^{111} \eqref{eqn-n-3-skein-first-map-21} R_{111}^{21}$ and  
$F_{12}^{111} \eqref{eqn-n-3-skein-first-map} R_{111}^{12}$.
    \item[(b)] The $D(\C_{111}\text{-}\C_{111})$ isomorphism between the convolutions of \eqref{eqn-O1-definition} and \eqref{eqn-O2-definition} which is induced via the isomorphisms provided by (a) from the isomorphisms 
$$ F_{21}^{111}F_{3}^{21} R_{21}^{3} R_{111}^{21} \xrightarrow{\alpha\beta^R} F_{12}^{111}F_{3}^{12}R_{12}^{3} R_{111}^{12}$$ 
is the requisite isomorphism intertwining the maps \eqref{eqn-map-O1-direct-sum} and \eqref{eqn-map-O2-direct-sum}.
\end{enumerate}

For (a), since \eqref{eqn-skein-relation21-no-crossing} fits into an exact triangle in $D(\C_{21}\text{-}\C_{21})$, there exist a homotopy equivalence of twisted complexes of the form
\begin{equation}
\label{eqn-from-F_21^111-F_3^21-R_21^3-R_111^21-to-S1O1R1-step1}
\begin{tikzcd}[column sep=4cm]
 \underset{\degzero}{F_{21}^{111}F_3^{21} R_{21}^{3} R_{111}^{21}}
 \ar{d}{F_{21}^{111} \eqref{eqn-n-3-skein-first-map-21} R_{111}^{21}}
 \ar{dr}{?}
 &
 \\
 \underset{\degzero}{F_{21}^{111} R_{111}^{21} F_{12}^{111} R_{111}^{12}  F_{21}^{111} R_{111}^{21}}
 \ar{r}{F_{21}^{111} \eqref{eqn-n-3-skein-second-map-21}R_{111}^{21}}
 &
 F_{21}^{111} R_{111}^{21}[-2]. 
\end{tikzcd}
\end{equation}
Since $F_{21}^{111}$ is split-spherical with cotwist [-3], there is a split exact triangle
$$ F_{21}^{111} R_{111}^{21}
\xrightarrow{F_{21}^{111} \action R_{111}^{21}}
 F_{21}^{111} R_{111}^{21} F_{21}^{111} R_{111}^{21}
\xrightarrow{\pi}
F_{21}^{111} R_{111}^{21}[-2].
$$
Therefore, the object $F_{21}^{111} R_{111}^{21}[-2]$ is homotopy equivalent to the twisted complex 
$ F_{21}^{111} R_{111}^{21}
\xrightarrow{F_{21}^{111} \action R_{111}^{21}}
 \underset{\degzero}{F_{21}^{111} R_{111}^{21} F_{21}^{111} R_{111}^{21}}$. It follows 
 by the Replacement Lemma that there exists a homotopy equivalence of form
 \begin{tiny}
 \begin{equation}
 \label{eqn-from-F_21^111-F_3^21-R_21^3-R_111^21-to-S1O1R1-step2}
\begin{tikzcd}[column sep=4.5cm]
 \underset{\degzero}{F_{21}^{111} R_{111}^{21} F_{12}^{111} R_{111}^{12}  F_{21}^{111} R_{111}^{21}}
 \ar{r}{F_{21}^{111} \eqref{eqn-n-3-skein-second-map-21}R_{111}^{21}}
 \ar{d}{\id \oplus 0}
 &
 F_{21}^{111} R_{111}^{21}[-2]
 \ar{d}{?}
 \ar{ld}{0 \oplus ?}
 \\
 \underset{\degzero}{F_{21}^{111} R_{111}^{21} F_{12}^{111} R_{111}^{12} F_{21}^{111} R_{111}^{21} \oplus F_{21}^{111} R_{111}^{21}}
 \ar{r}[']{F_{21}^{111} R_{111}^{21} \trace F_{21}^{111} R_{111}^{21} \oplus F_{21}^{111} \action R_{111}^{21}}
 &
 F_{21}^{111} R_{111}^{21} F_{21}^{111} R_{111}^{21}
\end{tikzcd}
\end{equation}
\end{tiny}
The composition of \eqref{eqn-from-F_21^111-F_3^21-R_21^3-R_111^21-to-S1O1R1-step1} and \eqref{eqn-from-F_21^111-F_3^21-R_21^3-R_111^21-to-S1O1R1-step2} is the desired homotopy equivalence from $F_{21}^{111}F_3^{21} R_{21}^{3} R_{111}^{21}$ to the twisted complex \eqref{eqn-O1-definition}. 

Therefore,  we have proven (a) for the twisted complex \eqref{eqn-O1-definition}; the proof for \eqref{eqn-O2-definition} is similar.

For (b), we notice that the target of the maps is the direct sum 
$$ F_{21}^{111} R_{111}^{21} F_{12}^{111} R_{111}^{12} \oplus F_{12}^{111} R_{111}^{12}F_{21}^{111} R_{111}^{21}. $$
We prove that the isomorphism intertwines the components of \eqref{eqn-map-O1-direct-sum} and \eqref{eqn-map-O2-direct-sum}
which go into the second direct summand. The proof that it intertwines the first direct summand components is similar. 

It suffices to show that the following diagram commutes in $D(\C_{111}\text{-}\C_{111})$:
\begin{tiny}
\begin{equation*}
\begin{tikzcd}[column sep=3cm, row sep=1cm]
F_{21}^{111}F_3^{21} R_{21}^{3} R_{111}^{21}
\ar{r}{\alpha\beta^R}
\ar[']{d}{F_{21}^{111} \eqref{eqn-n-3-skein-first-map-21} R_{111}^{21}}
&
F_{12}^{111}F_{3}^{12}R_{12}^{3} R_{111}^{12}
\ar{d}{F_{12}^{111} \eqref{eqn-n-3-skein-first-map} R_{111}^{12}}
\\
F_{21}^{111} R_{111}^{21} F_{12}^{111} R_{111}^{12}  F_{21}^{111} R_{111}^{21}
\ar[']{d}{\trace F_{12}^{111} R_{111}^{12}  F_{21}^{111} R_{111}^{21}}
&
F_{12}^{111} R_{111}^{12} F_{21}^{111} R_{111}^{21}  F_{12}^{111} R_{111}^{12}
\ar{d}{F_{12}^{111} R_{111}^{12} F_{21}^{111} R_{111}^{21} \trace}
\\
F_{12}^{111} R_{111}^{12}  F_{21}^{111} R_{111}^{21}
\ar[equals]{r}
&
F_{12}^{111} R_{111}^{12}  F_{21}^{111} R_{111}^{21}.
\end{tikzcd}
\end{equation*}
\end{tiny}
This can be simplified to:
\begin{tiny}
\begin{equation*}
\begin{tikzcd}[column sep=3cm, row sep=1cm]
F_{21}^{111}F_3^{21} R_{21}^{3} R_{111}^{21}
\ar{r}{\alpha\beta^R}
\ar[']{d}{F_{21}^{111}F_3^{21} R_{21}^{3} \action R_{111}^{21}}
&
F_{12}^{111}F_{3}^{12}R_{12}^{3} R_{111}^{12}
\ar{d}{F_{12}^{111} \action F_{3}^{12}R_{12}^{3} R_{111}^{12}}
\\
F_{21}^{111}F_3^{21} R_{21}^{3} R_{111}^{21} F_{21}^{111} R_{111}^{21}
\ar[']{d}{\alpha \beta^R F_{21}^{111} R_{111}^{21}}
&
F_{12}^{111} R_{111}^{12} F_{12}^{111} F_{3}^{12}R_{12}^{3} R_{111}^{12}
\ar{d}{F_{12}^{111} R_{111}^{12} \beta \alpha^R}
\\
F_{12}^{111} F_3^{12} R_{12}^{3} R_{111}^{12} F_{21}^{111} R_{111}^{21}
\ar[']{d}{F_{12}^{111} \trace R_{111}^{12} F_{21}^{111} R_{111}^{21}}
&
F_{21}^{111} R_{111}^{21} F_{21}^{111} F_{3}^{21}R_{12}^{3} R_{111}^{12}
\ar{d}{F_{21}^{111} R_{111}^{21} F_{21}^{111} \trace R_{111}^{12}}
\\
F_{12}^{111} R_{111}^{12}  F_{21}^{111} R_{111}^{21}
\ar[equals]{r}
&
F_{12}^{111} R_{111}^{12}  F_{21}^{111} R_{111}^{21}.
\end{tikzcd}
\end{equation*}
\end{tiny}
The commutativity of this diagram reduces to the commutativity of the
diagram
\begin{tiny}
\begin{equation*}
\begin{tikzcd}[column sep=1cm, row sep=1cm]
&
F_{12}^{111}F_3^{12} R_{21}^{3} R_{111}^{21}
\ar[']{ld}{F_{12}^{111}F_3^{12} R_{21}^{3} \action R_{111}^{21}}
\ar{rd}{F_{12}^{111} \action F_{3}^{12} R_{21}^{3} R_{111}^{21}}
&
\\
F_{12}^{111} F_3^{12} R_{21}^{3} R_{111}^{21} F_{21}^{111} R_{111}^{21}
\ar[']{d}{F_{12}^{111} F_3^{12} \beta^R F_{21}^{111} R_{111}^{21}}
&
&
F_{12}^{111} R_{111}^{12} F_{12}^{111} F_3^{12} R_{21}^{3} R_{111}^{21}
\ar{d}{F_{21}^{111} R_{111}^{21} \beta R_{21}^{3} R_{111}^{21}}
\\
F_{12}^{111} F_3^{12} R_{12}^{3} R_{111}^{12} F_{21}^{111} R_{111}^{21}
\ar[']{dr}{F_{12}^{111} \trace R_{111}^{12} F_{21}^{111} R_{111}^{21}}
&
&
F_{12}^{111} R_{111}^{12} F_{21}^{111} F_3^{21} R_{21}^{3} R_{111}^{21}
\ar{dl}{F_{12}^{111} R_{111}^{12} F_{21}^{111} \trace R_{111}^{21}}
\\
&
F_{12}^{111} R_{111}^{12}  F_{21}^{111} R_{111}^{21}.
&
\end{tikzcd}
\end{equation*}
\end{tiny}
and then further to the commutativity of 
\begin{tiny}
\begin{equation*}
\begin{tikzcd}[column sep=1.5cm, row sep=1cm]
&
F_{12}^{111}F_3^{12} R_{21}^{3} R_{111}^{21}
\ar[']{ld}{F_{12}^{111}F_3^{12} \beta^R}
\ar{rd}{\beta R_{21}^{3} R_{111}^{21}}
&
\\
F_{12}^{111} F_3^{12} R_{12}^{3} R_{111}^{12} 
\ar[']{dr}{\trace}
&
&
 F_{21}^{111} F_{3}^{21} R_{21}^{3} R_{111}^{21}
\ar{dl}{\trace}
\\
&
\id_{111},
&
\end{tikzcd}
\end{equation*}
\end{tiny}
which commutes since $\beta^R$ is the right dual of $\beta$.
\\
\\
\textbf{Flop+flop=twist}

We only prove the first of the ``flop-flop = twist'' relations, the second is proved identically. 
We need to show that 
\begin{equation}
\label{eqn-T2T1-initial-tensor}
\{F_{3}^{12} R_{21}^3 \xrightarrow{\mu} \underset{\degzero}{R_{111}^{12} F_{21}^{111}} \}
\{F_{3}^{21} R_{12}^3 \xrightarrow{\lambda} \underset{\degzero}{R_{111}^{21} F_{12}^{111}}\}
[2]
\end{equation}
is isomorphic to 
$$ F_{3}^{12} R_{12}^3 [-2] \xrightarrow{\psi}  F_{3}^{12} R_{12}^3 \xrightarrow{\trace} \id_{\C_{12}}. $$

The tensor of the product of convolutions in $\eqref{eqn-T2T1-initial-tensor}$ is isomorphic to 
the convolution of the twisted complex
\begin{small}
\begin{equation}
\label{eqn-T2T1-expanded-tensor}
\left(\begin{tikzcd}
F_{3}^{12} R_{21}^3 F_{3}^{21} R_{12}^3 
\ar{d}{\left( \begin{smallmatrix} \mu F_{3}^{21} R_{12}^3 \\ - F_{3}^{12} R_{21}^3 \lambda \end{smallmatrix} \right)}
\\
R_{111}^{12} F_{21}^{111} F_{3}^{21} R_{12}^3
\oplus
F_{3}^{12} R_{21}^3 R_{111}^{21} F_{12}^{111}
\ar{d}{\left( R_{111}^{12} F_{21}^{111} \lambda \quad \mu R_{111}^{21} F_{12}^{111} \right)}
\\
\underset{\degzero}{ R_{111}^{12} F_{21}^{111} R_{111}^{21} F_{12}^{111}}.
\end{tikzcd}\right)[2]
\end{equation}
\end{small}
Since $F_{21}^{111}$ and $F_{12}^{111}$ are split spherical with cotwist $[-3]$, we have 
$$ R_{111}^{12} F_{21}^{111} F_{3}^{21} R_{12}^3 \simeq R_{111}^{12} F_{12}^{111} F_{3}^{12} R_{12}^3 \simeq 
F_{3}^{12} R_{12}^3 \oplus F_{3}^{12} R_{12}^3[-2], $$ 
$$ F_{3}^{12} R_{21}^3 R_{111}^{21} F_{12}^{111} \simeq F_{3}^{12} R_{12}^3 R_{111}^{12} F_{12}^{111} \simeq 
F_{3}^{12} R_{12}^3 \oplus F_{3}^{12} R_{12}^3[-2], $$ 
Since $F_{3}^{21}$ is split $\mathbb{P}^2$-functor with $H=[-2]$ we also have 
$$ F_{3}^{12} R_{21}^3 F_{3}^{21} R_{12}^3 \simeq F_{3}^{12} R_{12}^3 \oplus F_{3}^{12} R_{12}^3[-2] \oplus F_{3}^{12} R_{12}^3 [-4].$$
Thus three out of four objects in the twisted complex \eqref{eqn-T2T1-expanded-tensor} are isomorphic to direct sums of shifted copies 
of $F_2R_2$. Under these identifications \eqref{eqn-T2T1-expanded-tensor} becomes:
\begin{small}
\begin{equation}
\label{eqn-T2T1-expanded-tensor-rewritten-with-direct-sums}
\begin{tikzcd}[row sep=0.9cm]
F_{3}^{12} R_{12}^3[2] \oplus F_{3}^{12} R_{12}^3 \oplus F_{3}^{12} R_{12}^3 [-2]
\ar{dd}{\begin{small}\begin{pmatrix}
\id \amsamp 0 \amsamp 0 \\
0 \amsamp \id \amsamp \psi_1 \\
-\id \amsamp 0 \amsamp  0\\
0 \amsamp -\id \amsamp - \psi_2
 \end{pmatrix}\end{small}}
\\
\\
\left( F_{3}^{12} R_{12}^3[2] \oplus F_{3}^{12} R_{12}^3 \right)
\oplus
\left(  F_{3}^{12} R_{12}^3[2] \oplus F_{3}^{12} R_{12}^3  \right)
\ar{d}{\begin{small}\begin{pmatrix}
\eqref{eqn-n-3-skein-first-map} \amsamp \phi_1 \amsamp
\eqref{eqn-n-3-skein-first-map} \amsamp \phi_2
\end{pmatrix}\end{small}}
\\
\underset{\degzero}{ R_{111}^{12} F_{21}^{111} R_{111}^{21} F_{12}^{111}}[2]
\end{tikzcd},
\end{equation}
\end{small}
where $\psi_1, \psi_2$ are the compositions 
$$ F_{3}^{12} R_{12}^3 [-2] \hookrightarrow F_{3}^{12} R_{12}^3 F_{3}^{12} R_{12}^3 \xrightarrow{\trace F_{3}^{12} R_{12}^3, \quad F_{3}^{12} R_{12}^3\trace}  
F_{3}^{12} R_{12}^3,$$ 
and $\phi_1$ and $\phi_2$ are the compositions
\begin{small}
\begin{equation*}
\begin{tikzcd}
F_{3}^{12} R_{12}^3
\ar[hook]{d}
\\
R_{111}^{12} F_{12}^{111} F_{3}^{12} R_{12}^3[2]
\ar{d}{R_{111}^{12} (\text{multifork}) R_{12}^3}
\\
R_{111}^{12} F_{21}^{111} F_{3}^{21} R_{12}^3[2]
\ar{d}{R_{111}^{12} F_{21}^{111} \;\action\; F_{3}^{21} R_{12}^3}
\\
R_{111}^{12} F_{21}^{111} R_{111}^{21} F_{21}^{111} F_{3}^{21} R_{12}^3[2]
\ar{d}{R_{111}^{12} F_{21}^{111} R_{111}^{21}  (\text{multifork}) R_{12}^3}
\\
R_{111}^{12} F_{21}^{111} R_{111}^{21} F_{12}^{111} F_{3}^{12} R_{12}^3[2]
\ar{d}{R_{111}^{12} F_{21}^{111} R_{111}^{21} F_{12}^{111} \;\trace}
\\
R_{111}^{12} F_{21}^{111} R_{111}^{21} F_{12}^{111} [2]
\end{tikzcd} 
\text{ and }
\begin{tikzcd}
F_{3}^{12} R_{12}^3
\ar[hook]{d}
\\
F_{3}^{12} R_{12}^3 R_{111}^{12} F_{12}^{111} [2]
\ar{d}{R_{111}^{12} (\text{multifork}) R_{12}^3}
\\
F_{3}^{12} R_{21}^3 R_{111}^{21} F_{12}^{111} [2]
\ar{d}{F_{3}^{12} R_{21}^3 \;\action\; R_{111}^{21} F_{12}^{111}}
\\
F_{3}^{12} R_{21}^3 R_{111}^{21} F_{21}^{111} R_{111}^{21} F_{12}^{111}[2]
\ar{d}{ F_{3}^{12} (\text{multifork}) F_{21}^{111} R_{111}^{21} F_{12}^{111}}
\\
F_{3}^{12} R_{12}^3 R_{111}^{12} F_{21}^{111} R_{111}^{21} F_{12}^{111}[2]
\ar{d}{\trace\; R_{111}^{12} F_{21}^{111} R_{111}^{21} F_{12}^{111}}
\\
R_{111}^{12} F_{21}^{111} R_{111}^{21} F_{12}^{111}[2]
\end{tikzcd} 
\end{equation*}
\end{small}
Note that since the map \eqref{eqn-n-3-skein-second-map} is the composition 
$$ R_{111}^{12} F_{21}^{111} R_{111}^{21} F_{12}^{111}
\xrightarrow{R_{111}^{12} \trace F_{12}^{111}} 
R_{111}^{12} F_{12}^{111}
\twoheadrightarrow{\id_{111}[-2]}, $$
we have the following equality of maps $F_{3}^{12} R_{12}^3
\rightarrow \id_{3}$:
$$ \phi_1 \circ \eqref{eqn-n-3-skein-second-map} = \trace = 
\phi_2 \circ \eqref{eqn-n-3-skein-second-map}. $$

The map $\psi$ in the definition of the $\mathbb{P}$-twist of $F_{3}^{12}$ is $\psi_1 - \psi_2$. 
Changing the basis of the middle term of \eqref{eqn-T2T1-expanded-tensor-rewritten-with-direct-sums} 
to the diagonal and the antidiagonal we obtain:
\begin{small}
\begin{equation}
\label{eqn-T2T1-expanded-tensor-rewritten-with-direct-sums-after-change-of-basis}
\begin{tikzcd}[row sep=0.9cm]
F_{3}^{12} R_{12}^3[2] \oplus F_{3}^{12} R_{12}^3 \oplus F_{3}^{12} R_{12}^3 [-2]
\ar{dd}{\begin{small}\begin{pmatrix}
0 \amsamp 0 \amsamp 0 \\
0 \amsamp 0\amsamp \psi \\
2\id \amsamp 0 \amsamp  0\\
0 \amsamp 2\id \amsamp \psi_1 + \psi_2
 \end{pmatrix}\end{small}}
\\
\\
\left( F_{3}^{12} R_{12}^3[2] \oplus F_{3}^{12} R_{12}^3 \right)
\oplus
\left(  F_{3}^{12} R_{12}^3[2] \oplus F_{3}^{12} R_{12}^3  \right)
\ar{d}{\begin{small}\begin{pmatrix}
2 \eqref{eqn-n-3-skein-first-map} \amsamp \phi_1 + \phi_2 \amsamp  0
\amsamp \phi_1 - \phi_2 
\end{pmatrix}\end{small}}
\\
\underset{\degzero}{ R_{111}^{12} F_{21}^{111} R_{111}^{21} F_{12}^{111}}[2].
\end{tikzcd}
\end{equation}
\end{small}
We can remove the following acyclic subcomplex of \eqref{eqn-T2T1-expanded-tensor-rewritten-with-direct-sums-after-change-of-basis}
$$
F_{3}^{12} R_{12}^3[2] \oplus F_{3}^{12} R_{12}^3
\xrightarrow{ \begin{small}\begin{pmatrix}
2\id \amsamp 0
\\
0 \amsamp 2\id 
\end{pmatrix}\end{small}}
\underset{\degminusone}{F_{3}^{12} R_{12}^3[2] \oplus F_{3}^{12} R_{12}^3}
$$
using the Replacement Lemma. 
Since the subcomplex has no external arrows emerging from its degree $-2$ part, 
no other differentials are affected. We obtain:
\begin{small}
\begin{equation}
\label{eqn-T2T1-expanded-tensor-rewritten-with-direct-sums-first-extraction}
\begin{tikzcd}[column sep=0.9cm]
F_{3}^{12} R_{12}^3 [-2]
\ar{r}{\begin{small}\begin{pmatrix}
0 \\
\psi \\
\end{pmatrix}\end{small}}
&
F_{3}^{12} R_{12}^3[2] \oplus F_{3}^{12} R_{12}^3 
\ar{rr}{\begin{small}\begin{pmatrix}
2 \eqref{eqn-n-3-skein-first-map} \amsamp \phi_1 + \phi_2 
\end{pmatrix}\end{small}}
& &
\underset{\degzero}{ R_{111}^{12} F_{21}^{111} R_{111}^{21} F_{12}^{111}}[2].
\end{tikzcd}
\end{equation}
\end{small}
Since \eqref{eqn-skein-relation-no-crossing} fits into an exact
triangle, we have a homotopy equivalence of form 
\begin{equation}
\begin{tikzcd}
F_{3}^{12} R_{12}^3[2]  
\ar{r}{2 \eqref{eqn-n-3-skein-first-map}} 
\ar{dr}{*}
&
R_{111}^{12} F_{21}^{111} R_{111}^{21} F_{12}^{111}[2] 
\ar{d}{\frac{1}{2} \eqref{eqn-n-3-skein-second-map}}
\\
&
\underset{\degzero}{\id_{12}}
\end{tikzcd}
\end{equation}
By the Replacement Lemma, we finally obtain:
\begin{small}
\begin{equation}
\label{eqn-T2T1-expanded-tensor-rewritten-with-direct-sums-final}
\begin{tikzcd}[column sep=1.25cm]
F_{3}^{12} R_{12}^3 [-2]
\ar{r}{ \psi }
&
F_{3}^{12} R_{12}^3[2] 
\ar{rr}{\frac{1}{2} \eqref{eqn-n-3-skein-second-map} \circ (\phi_1 +
\phi_2)}
& &
\underset{\degzero}{\id_{12}}. 
\end{tikzcd}
\end{equation}
\end{small}
Finally, since $\eqref{eqn-n-3-skein-second-map} \circ \phi_i = \trace$, 
this is homotopy equivalent to the $\mathbb{P}$-twist of $F_{3}^{12}$:
\begin{small}
\begin{equation}
\label{eqn-T2T1-ptwist}
\begin{tikzcd}[column sep=0.9cm]
F_{3}^{12} R_{12}^3 [-2]
\ar{r}{ \psi }
&
F_{3}^{12} R_{12}^3[2] 
\ar{rr}{\trace}
& &
\underset{\degzero}{\id_{12}}. 
\end{tikzcd}
\end{equation}
\end{small}

\textbf{$(T_{21}^{12},D_{12}^{21})$ and $(T_{21}^{12},D_{12}^{21})$ are pairs of mutually inverse equivalences.}

This follows from the ``twist-twist=flop'' relations.

Indeed,  
$T_{21}^{12} T_{12}^{21}$ and $T_{12}^{21} T_{21}^{12}$ are isomorphic to $\mathbb{P}$-twists of $F_3^{12}$ and $F_3^{21}$ and hence 
are both autoequivalences. Therefore $T_{12}^{21}$ and $T_{21}^{12}$ are also autoequivalences. On the other hand, 
the maps $\lambda'$ and $\mu'$ which define the functors $D_{12}^{21}$ and $D_{21}^{12}$ are the left duals 
of the maps $\mu$ and $\lambda$ which define the functors $T_{21}^{12}$ and $T_{12}^{21}$. Hence $D_{12}^{21}$ and $D_{21}^{12}$
are the left adjoints of $T_{21}^{12}$ and $T_{12}^{21}$. 
\end{proof}

\section{A  split-$\mathbb{P}^2$ action of $ \gbrcat_3 $ on $T^*\Fl_3(\bar{i})$ }
\label{flag_act_sec}

The aim of the second part of the paper is to define a network of categories and functors that satisfies the assumptions of Theorem \ref{Theorem1}.

In particular, we construct such a network on $T^*\Fl_3(\bar{i})$ and we prove it verifies all the hypothesis for a  split $\mathbb{P}^2$ representation of $GBr_3$.

\subsection{Notations and useful results}

In this section we fix some notation and recall some useful results.

In the following, let $D^b(X)$ be the derived category of coherent sheaves on a smooth quasi-projective variety $X$ and assume as before all the functors derived, i.e. we omit $\mathbf{R}$ and $\mathbf{L}$.

We omit the pushforward $i_*$ applied to structure sheaves when $i$ is an embedding and the pullbacks applied to line bundles. We write $\mathcal{O}_{X\times_Z Y} (D_1, D_2)$ for the line bundle $\mathcal{O}_X (D_1)\boxtimes \mathcal{O}_Y (D_2) $ on the fiber product $X \times_Z Y$ when $D_1$ and $D_2$ are divisors respectively in $X$ and $Y$.

Recall (see section  1.2 of \cite{CG}) the description the total space $T^* \Fla _n (\bar{i} ) $ of the cotangent bundle of the  flag $\Fl_n(\bar{i})$ as a Springer resolution:

\[  \cota \Fla _n (\bar{i} ) =  \left\{(V_{\bar{i}} ,\alpha)\;  \middle \vert \; \alpha:\mathbb{C}^n \to \mathbb{C}^n ; \alpha(V_{i_k}) \subset V_{i_{k-1}}  \right\}. \]

With a little abuse of notation we write

\begin{align*}
 T^* \Fla _n (\bar{i} )=\left\{
\xymatrix@C=0.1em{
0 &\subset & V_{i_1} \ar_{\alpha}@/_1em/[ll] & \subset  & \dots \ar_{\alpha}@/_1em/[ll]& \subset & V_{i_{k-1}} \ar_{\alpha}@/_1em/[ll] & \subset &\mathbb{C}^n
\ar_{\alpha}@/_1em/[ll]}  \right\}.
\end{align*}
Similarly, we denote with

\begin{align*}
 \left\{
\xymatrix@C=0.1em{
0 &\subset & V_{\lambda_1} \ar_{\alpha}@/_1em/[ll] & \subset  & \dots \ar_{\alpha}@/_1em/[ll]& \subset & V_{\lambda_{k-1}} \ar_{\alpha}@/_1em/[ll] & \subset &\mathbb{C}^n
\ar_{\alpha}@/_1em/[ll]}  \right\}
\times \left\{
\xymatrix@C=0.1em{
0 &\subset & V_{j_1} \ar_{\alpha}@/_1em/[ll] & \subset  & \dots \ar_{\alpha}@/_1em/[ll]& \subset & V_{j_{h-1}} \ar_{\alpha}@/_1em/[ll] & \subset &\mathbb{C}^n
\ar_{\alpha}@/_1em/[ll]}  \right\}
\end{align*}

the subspace of $T^* \Fla _n (\bar{i} ) \times T^* \Fla _n (\bar{j} )$ where the maps $\alpha$s need to satisfy  $\alpha(V_{i_k}) \subset V_{i_{k-1}} $ and $\alpha(V_{j_k}) \subset V_{j_{k-1}} $.

We use  $\mathcal{V}_i$ for the pullback on $\Fl_n$ of the tautological bundle of $\mathrm{Gr}(i,n)$.

Recall that if $X$ and $Y$ are quasi-projective subvarieties of $Z$ such that their intersection $X\cap Y$  is a Cartier divisor in $Y$ then we have the short exact sequence of coherent sheaves of $Z$
\begin{equation}\label{SES}
    0 \to \mathcal{O}_Y(-X\cap Y) \to \mathcal{O}_{X\cup Y} \to  \mathcal{O}_X \to 0.
\end{equation}

Moreover, note that if $f: X \to Y$ is proper, then we have the following adjunction of derived functors 

\begin{equation}\label{adjunction}
 f^* \dashv  f_* \dashv f^! 
\end{equation}
where  $f^! = f^*(-)\otimes \omega_{X/Y} [dim Y- dim X]$.

In what follows, we make extensive use of the theory of Fourier-Mukai transforms (\cite{Hu} as reference).

If $f$ is a divisorial embedding we  write $\mathcal{O}_{X\times_X Y} \in D^b(X\times Y)$ and $\mathcal{O}_{Y \times_X X} \in D^b(Y \times X)(X,0)[-1]$ for the Fourier-Mukai kernels representing respectively $f_*$ and $f^!$.

If $f$ is a fibration we  write $\mathcal{O}_{X\times_X Y} \in D^b(X\times Y)$ and $\mathcal{O}_{Y \times_X X} \in D^b(Y \times X)$ for the Fourier-Mukai kernels representing respectively $f_*$ and $f^*$.

We write moreover $\mathcal{O}_{X\times_X X} \in D^b(X\times X)$ for the Fourier-Mukai kernel representing the identity on $X$.

Furthermore, recall how the composition of two Fourier-Mukai transforms works.
Let $X,Y,Z$ be separated schemes of finite type over $k$,  $E_1 \in D_{qcoh}(X \times Y)$ and $E_2\in D_{qcoh}(Y \times Z)$.

The composition of the Fourier-Mukai trasforms $\Phi_{E_2} \circ \Phi_{E_1}$ is isomorphic to the Fourier-Mukai transform $\Phi_{E_1\star E_2}$ with kernel the convolution of the kernels $E_1$, $E_2$ defined

\begin{equation} \label{FM-comp}
E_2 \star E_1 = \pi_{13*} ( \pi_{12}^* E_1 \otimes \pi_{23}^* E_2 )
\end{equation}

where $\pi_{12}$, $\pi_{23}$ and $ \pi_{13} $ are the natural projections

\[
\begindc{\commdiag}[250]

\obj(1,3)[d]{$X\times Y$}
\obj(3,3)[e]{$X \times Z$}
\obj(5,3)[f]{$Y \times Z$}
\obj(3,5)[g]{$X \times Y \times Z$}
\mor{g}{d}{$\pi _{12}$}[-1,0]
\mor{g}{e}{$\pi _{13} $}
\mor{g}{f}{$\pi _{23} $}

\enddc
\]
(See section 5.1 of \cite{Hu}).

In order to deal with multiple convolutions, we recall some technology involving the derived tensor product of two structure sheaves.

Let $Z_1,Z_2$ be two locally complete intersection subvarieties of a smooth algebraic variety $Z$ with their intersection $W=Z_1 \cap Z_2$ being a locally complete subvariety of $Z_1$ and $Z_2$.

The excess bundle $\mathcal{E}_W$ of the intersection $W=Z_1 \cap Z_2$, is the locally free sheaf which fits in the short exact sequence of sheaves on $W$
\begin{equation}
    0 \longrightarrow \mathcal{N}_{W/Z} \longrightarrow j_1^*  \mathcal{N}_{Z_1/Z} \oplus j_2^* \mathcal{N}_{Z_2/Z} \longrightarrow \mathcal{E}_W \longrightarrow 0
\end{equation}
where $j_i$, $i=1,2$, is the inclusion of of $W$ in $Z_i$.

Under these hypotheses, the cohomology sheaves of  $i_2^* i_{1*}\mathcal{O}_{Z_1}\in D^b(Z_2)$ are 
\[H^{-q}  (i_2^* i_{1*} \mathcal{O}_{Z_1})= j_{1*}( \bigwedge ^q \mathcal{E}_W^\vee )\] 
where $j_i$, $i=1,2$, is the inclusion of $Z_1$ in the ambient space $Z$ (See \cite{Sc}).

In particular, if the intersection $W = Z_1  \cap Z_2 $ is transverse then we have the following isomorphism

\begin{equation} \label{transverse}
\mathcal{O}_{Z_1}  \otimes \mathcal{O}_{Z_2} \simeq \mathcal{O}_W.
\end{equation}

\subsection{Our Setup}

Let $C$ be our ambient variety the total space of the cotangent bundle  of $Fl_3$
\[C= T^* Fl_3. \]

Let $A$ and $E$ be the quasi-projective varieties defined as the total space of the cotangent bundle of $\mathbb{P}^2$ and   $\mathbb{P}^{2\vee}$:

\[ A= T^* \mathbb{P}^2 \qquad E = T^* \mathbb{P}^{2\vee}. \]

Let $B$ and $D$ be the quasi-projective varieties defined as the total space of the pullback via $p_1$ and $p_2$ on $Fl_3$ of the cotangent bundle of $\mathbb{P}^2$ and   $\mathbb{P}^{2\vee}$:

\[ B= p_1^* T^* \mathbb{P}^2 \qquad D = p_2^* T^* \mathbb{P}^{2\vee}. \]
We have therefore the following diagram

\begin{equation*} \label{pflag4}
\begindc{\commdiag}[150]
\obj(1,1)[a]{$A$}
\obj(1,4)[b]{$B$}
\obj(6,4)[c]{$C$}
\obj(10,4)[d]{$D$}
\obj(10,1)[e]{$E$}
\obj(1,-2)[f]{$ \mathbb{P}^2 $}
\obj(10,-2)[g]{$ \mathbb{P}^{2\vee} $}
\obj(6,-2)[h]{$ pt $}

\mor{b}{a}{$\pi _A$}[0,8]
\mor{b}{c}{$i_B$}[-1,6]
\mor{d}{c}{$i_D$}[0,6]
\mor{d}{e}{$\pi _E$}[0,8]
\mor{f}{a}{$i_{\mathbb{P}^{2}}$}[-1,6]
\mor{g}{e}{$i_{\mathbb{P}^{2\vee}}$}[0,6]
\mor{g}{h}{$\pi_{pt1}$}[0,8]
\mor{f}{h}{$\pi_{pt2}$}[-1,8]
\enddc
\end{equation*}

where $i_B$, $i_D$,$i_{\mathbb{P}^{2\vee}}$ and $i_{\mathbb{P}^{2}}$ are divisorial inclusions and $B \xrightarrow{\pi_A} A$ and  $D \xrightarrow{\pi_E} E$ are $\mathbb{P}^1$ bundles.

\begin{definition}\label{categories}
We define the categories $\C_{(3)}=D^b(pt)$, $\C_{(1,2)}=D^b(A)$, $\C_{(2,1)}=D^b(E)$, and $\C_{(1,1,1)}=D^b(C)$.
\end{definition}

Notice that we have the following descriptions of  the quasi-projective varieties

\begin{align*}
 A= \left\{
\xymatrix@C=0.1em{
0 &\subset & V_1 \ar_{\alpha}@/_2em/[ll] & \subset   & \mathbb{C}^3 \ar_{\alpha}@/_2em/[ll] } ;\ \dim (V_1)=1 \right\},
\end{align*} 
and
\begin{align*}
E=\left\{
\xymatrix@C=0.1em{
0  & \subset  & V_2 \ar_{\alpha}@/_2em/[ll]& \subset & \mathbb{C}^3  \ar_{\alpha}@/_2em/[ll] } ;\ \dim (V_2)=2 \right\} .
\end{align*} 
 
Moreover the four dimensional subvariety of $C$  defined as the (transverse) intersection of $B$ and $D$ can be described as 
 \begin{align*}
B \cap D= \left\{
\xymatrix@C=0.1em{
0 &\subset V_1 \subset  & V_2 \ar_{\alpha}@/_2em/[ll]& \subset & \mathbb{C}^3  \ar_{\alpha}@/_-2em/[lll] }  \right\}.
\end{align*} 

equipped with the two natural forgetful maps

\begin{equation*}
\begindc{\commdiag}[200]
\obj(0,0)[a]{$ A $}
\obj(2,0)[b]{$ E $}
\obj(1,2)[c]{$ B \cap D $}
\mor{c}{a}{$ q_1 $}[-1,0]
\mor{c}{b}{$ q_2 $}[1,0]
\enddc 
\end{equation*}

where $q_i$ is the map that forgets the choice of the $n-i$-th space.

Both $q_1$ and $q_2$ are isomorphic to the blow-up of the zero section carved out by $\{ \alpha =0  \}$ in respectively $A$ and $E$.

Both the bow-ups have the same exceptional divisor $Fl_3$ carved out by $\{ \alpha =0  \}$ and  the resulting birational transformations
\[ q_2 \circ q_1^{-1} : A \dasharrow E \]
\[ q_1 \circ q_2^{-1} : E \dasharrow A \]
 are a local model of a four dimensional Mukai flop.

\section{A split-$\mathbb{P}^2$ action of $\gbrcat_3$ on $T^*\Fl_3(\bar{i})$: generators}
\label{cat_act_flag_gen_sec}
In this section, we  define the generators of split $\mathbb{P}^2$ action of $GBr_3$, the forks and the merges.

After describing them as Fourier-Mukai transforms, we compute some of their convolutions which will be useful for our purpose.

In the following sections, we use $D^b(pt)$, $D^b(A)$, $D^b(E)$, $D^b(C)$ for our setup instead of the notation of Definition \ref{categories} for convenience.  

\begin{definition}
We define the  first fork functor 
\begin{equation} 
 F^{111}_{21} = i_{D*} \circ  \pi_E^*: D^b(E) \to D^b (C) 
\end{equation}
and define the second fork functor  
\begin{equation} 
 F^{111}_{12} = i_{B*} \circ  \pi_A^*: D^b(A) \to D^b (C). 
\end{equation}

 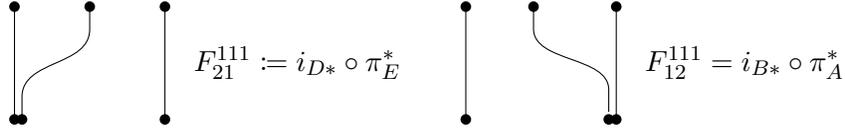
\begin{figure}[!h] \centering 
\begin{tikzpicture}

\braid [number of strands=3,
style strands={1,3}{opacity=0}, height=1 cm, width=0.9cm]at (2.1,-1) a_1^{-1};
\braid [number of strands=5,
style strands={2,4,5}{opacity=0}, height=1 cm, width=1 cm]at (2,-1)  a_4;

\fill[black] (2,-1) circle (2pt);
\fill[black] (4,-1) circle (2pt);
\fill[black] (2.1,-2.5) circle (2pt);
\fill[black] (2,-2.5) circle (2pt);
\fill[black] (3,-1) circle (2pt);
\fill[black] (4,-2.5) circle (2pt);

\draw  node at (4.25,-1.75) [right]   {$ F^{111}_{21} \coloneqq i_{D*} \circ  \pi_E^*  $};

\braid [number of strands=3,
style strands={1,3}{opacity=0}, height=0.9 cm, width=1cm]at (7.9,-1) a_2;
\braid [number of strands=5,
style strands={2,4,5}{opacity=0}, height=1 cm, width=1 cm]at (8,-1)  a_4;

\fill[black] (8,-1) circle (2pt);
\fill[black] (10,-1) circle (2pt);
\fill[black] (8,-2.5) circle (2pt);
\fill[black] (8.9,-1) circle (2pt);
\fill[black] (9.9,-2.5) circle (2pt);
\fill[black] (10,-2.5) circle (2pt);

\draw  node at (10.25,-1.75) [right]   {$F^{111}_{12} = i_{B*} \circ  \pi_A^* $};

\end{tikzpicture}
            \caption{First and second forks } 
    \end{figure}

\end{definition}

\begin{proposition}
The Fourier-Mukai kernel of $F^{111}_{21}$ and $F^{111}_{12}$ are respectively the sheaves $ \mathcal{O}_{E\times_E D} \in D^b(E \times C ) $ and $ \mathcal{O}_{A \times_A B} \in D^b(A \times C ) $.
\end{proposition}
\begin{proof}
Let  $\pi_{12}, \pi_{23}, \pi_{13}$ be the natural projections
\begin{equation}\label{splittingprojections}
\begin{tikzcd}
  &  E \times D \times C \arrow[dl, "\pi_{12}"' ]\arrow[d, "\pi_{13}" ]\arrow[dr, "\pi_{23}" ]  &\\
E\times D    &  D\times C    &  D\times E  \\
\end{tikzcd}.
\end{equation}

By (\ref{FM-comp}), the composition $F^{111}_{21}$ of $i_{D*} \circ  \pi_E^*$ is represented by  the convolution of their Fourier-Mukai kernels, thus
\[  \pi_{13*}( \pi_{12}^* \mathcal{O}_{E \times_E E}  \otimes  \pi_{23}^*  \mathcal{O}_{D \times_D D} )= \pi_{13*}( \mathcal{O}_{ E\times_E E \times C}   \otimes     \mathcal{O}_{E \times D\times_D D } )\]

By (\ref{transverse}), the latter derived tensor product of the  structure sheaves  is isomorphic to the structure sheaf of the intersection
\[ \mathcal{O}_{ E\times_E E \times C}   \otimes     \mathcal{O}_{E \times D\times_D D } \simeq \mathcal{O}_{E \times_E E \times_E D}  \]

Indeed the subvariety $E\times_E D \times C $ is of codimension 4 inside $E \times D \times C$ while the codimension of the subvariety $ E \times D \times_D D$ is 6 inside $E \times D \times C$.

Their intersection $(E\times_E D \times C)\cap ( E \times D\times_D D)$ is smooth and of codimension 9, therefore they intersect transversally.

Notice that $E \times_E E \times_ E D$ is isomorphic to a copy of $D$ in the third component, therefore $\pi_{13}:E \times _E E \times _E D \to  E\times C$ is an embedding, so the derived pushforward of $\mathcal{O}_{E \times _E E \times _E D }$ is just the structure sheaf of the image of the support.

So in conclusion the Fourier-Mukai kernel of $F^{111}_{21}$ is isomorphic to
\[  \mathcal{O}_{E\times_E D}. \]

The same argument applies to $i_{B*} \circ  \pi_A^*$ for showing that the kernel of $F^{111}_{12}$ is isomorphic to
\[  \mathcal{O}_{A \times_A B}. \]
\end{proof}

\begin{remark}
The first merge 
\begin{equation}
 R^{21}_{111} = \pi_{E*} \circ  i_D^!: D^b(C) \to D^b (E) 
\end{equation}
  and  the second merge
\begin{equation}
 R^{12}_{111} = \pi_{A*} \circ  i_B^!: D^b(C) \to D^b (A) 
\end{equation}
are respectively the right adjoints of  $F_{21}^{111} $ and  $F_{12}^{111} $.
\begin{figure}[!h] \centering \begin{tikzpicture}

\braid [number of strands=3,
style strands={2,3}{opacity=0}, height=1 cm, width=0.9cm]at (2.1,-1) a_1;
\braid [number of strands=5,
style strands={2,4,5}{opacity=0}, height=1 cm, width=1 cm]at (2,-1)  a_4;

\draw  node at (4.25,-1.75) [right]   {$R^{21}_{111} = p_{E*} \circ  i_D^!  $};

\fill[black] (2,-1) circle (2pt);
\fill[black] (2.1,-1) circle (2pt);
\fill[black] (4,-1) circle (2pt);
\fill[black] (2,-2.5) circle (2pt);
\fill[black] (3,-2.5) circle (2pt);
\fill[black] (4,-2.5) circle (2pt);

\braid [number of strands=3,
style strands={1,2}{opacity=0}, height=1 cm, width=0.9cm]at (8.1,-1) a_2^{-1};
\braid [number of strands=5,
style strands={2,4,5}{opacity=0}, height=1 cm, width=1 cm]at (8,-1)  a_4^{-1};

\fill[black] (8,-1) circle (2pt);
\fill[black] (9.9,-1) circle (2pt);
\fill[black] (10,-1) circle (2pt);

\fill[black] (8,-2.5) circle (2pt);
\fill[black] (9,-2.5) circle (2pt);
\fill[black] (10,-2.5) circle (2pt);

\draw  node at (10.25,-1.75) [right]   {$R_{111}^{12} =  p_{A*} \circ  i_B^!  $};

\end{tikzpicture}  \caption{First and second merges} \end{figure}
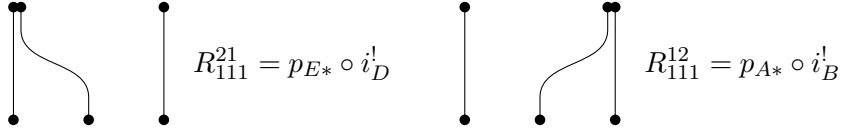

\end{remark}

\begin{proposition}\label{Splitkernel}
The Fourier-Mukai kernel associated to $R^{21}_{111}$ and $R^{12}_{111}$ are respectively the objects $ \mathcal{O}_{D \times_E E}(D,0)[-1] \in D^b(C \times E ) $ and $ \mathcal{O}_{B \times_A A}(B,0)[-1] \in D^b(C \times A ) $.
\end{proposition}

\begin{proof}
Let  $\pi_{12}, \pi_{23}, \pi_{13}$ be the natural projections as in diagram (\ref{splittingprojections}).

By (\ref{FM-comp}), the composition $R^{21}_{111} $ of $p_{E*} \circ  i_D^!$ is represented by  the convolution  of their Fourier-Mukai kernels, thus
\[  \pi_{13*}(\pi_{23}^*  \mathcal{O}_{D \times_D D} (D,0)  \otimes \pi_{12}^* \mathcal{O}_{E \times_E E} )  \simeq \pi_{13*}(\mathcal{O}_{E \times D\times_D D }   (D,0,0) \otimes   \mathcal{O}_{ E\times_E E \times C}    ) \] 

As in the proof of Proposition (\ref{splittingprojections}), by transversality of the intersection of $E \times D\times_D D$ with  $E\times_E E \times C$, we can apply (\ref{transverse}) and obtain the isomorphism.

\[ \mathcal{O}_{E \times D\times_D D }   (D,0,0) \otimes   \mathcal{O}_{ E\times_E E \times C}   ) \simeq \mathcal{O}_{E \times_E E \times_E D} (D,0,0) .\]

Notice that $E \times_E E \times_ E D$ is isomorphic to a copy of $D$ in the third component, therefore $\pi_{13}:E \times _E E \times _E D \to  E\times C$ is an embedding, so the derived pushforward of $\mathcal{O}_{E \times _E E \times _E D }$ is just the structure sheaf of the image of the support.

Since  $\pi_{13}$ is an embedding restricted to $E \times _E E \times _E D$, we conclude that the Fourier-Mukai kernel associated to $R^{21}_{111}$ is isomorphic to 
\[ \mathcal{O}_{D\times_E E}(D,0)[-1]. \]

The same argument applies to $i_{B*} \circ  \pi_A^*$ for showing that the Fourier-Mukai associated to $R^{12}_{111}$ is isomorphic to
\[  \mathcal{O}_{B \times_A A}(B,0)[-1]. \]

\end{proof}

\begin{definition}
We define the third and fourth forks
\[ F_{3}^{12}\coloneqq i_{\mathbb{P}^{2}*} \circ  \pi_{pt1}^*: D^b(pt) \to D^b (A), \quad F_{3}^{21}\coloneqq i_{\mathbb{P}^{2\vee}*} \circ  \pi_{pt2}^*: D^b(pt) \to D^b (E)  \] 
and  the third and fourth merges
\[  R^{3}_{12} = \pi_{pt1*} \circ  i_{\mathbb{P}^{2}}^!: D^b(A) \to D^b (pt) , \quad  R^{3}_{21} = \pi_{pt2*} \circ  i_{\mathbb{P}^{2\vee}}^!: D^b(E) \to D^b (pt).\]
\end{definition}

 \begin{figure}[!h] \centering 
\begin{tikzpicture}

 \braid [number of strands=3,
style strands={1,3}{opacity=0}, height=1 cm, width=0.9cm]at (2.1,-1) a_1^{-1};
\braid [number of strands=3,
style strands={1,3}{opacity=0}, height=1 cm, width=0.9cm]at (2.2,-1) a_1^{-1};
\braid [number of strands=5,
style strands={2,3,4,5}{opacity=0}, height=1 cm, width=1 cm]at (2,-1)  a_4;
\fill[black] (2,-1) circle (2pt);
\fill[black] (2.2,-2.5) circle (2pt);
\fill[black] (2.1,-2.5) circle (2pt);
\fill[black] (2,-2.5) circle (2pt);
\fill[black] (3,-1) circle (2pt);
\fill[black] (3.1,-1) circle (2pt);

\draw  node at (4.25,-1.75) [right]   {$  F_{3}^{12}\coloneqq i_{\mathbb{P}^{2}*} \circ  \pi_{pt1}^* $};

\braid [number of strands=3,
style strands={1,3}{opacity=0}, height=1 cm, width=0.9cm]at (7.6,-1) a_2;
\braid [number of strands=3,
style strands={1,3}{opacity=0}, height=1 cm, width=0.9cm]at (7.5,-1) a_2;
\braid [number of strands=5,
style strands={2,1,4,5}{opacity=0}, height=1 cm, width=1 cm]at (7.5,-1)  a_4;
\fill[black] (9.5,-1) circle (2pt);
\fill[black] (9.3,-2.5) circle (2pt);
\fill[black] (9.4,-2.5) circle (2pt);
\fill[black] (9.5,-2.5) circle (2pt);
\fill[black] (8.5,-1) circle (2pt);
\fill[black] (8.4,-1) circle (2pt);

\draw  node at (10.25,-1.75) [right]   {$ F_{3}^{21}\coloneqq i_{\mathbb{P}^{2\vee}*} \circ  \pi_{pt2}^* $};

\end{tikzpicture}
            \caption{Third and fourth forks} 
    \end{figure}
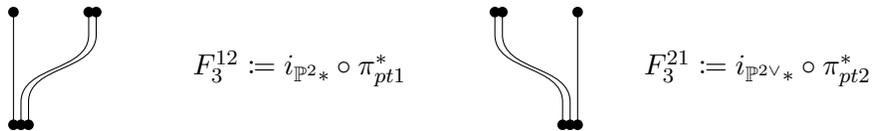

 \begin{figure}[!h] \centering 
 \adjustbox{scale=1,center}{
\begin{tikzpicture}

\braid [number of strands=3,
style strands={2,3}{opacity=0}, height=1 cm, width=0.9cm]at (2.1,0) a_1;
\braid [number of strands=3,
style strands={2,3}{opacity=0}, height=1 cm, width=0.9cm]at (2.2,0) a_1;
\braid [number of strands=5,
style strands={2,3,4,5}{opacity=0}, height=1 cm, width=1 cm]at (2,0)  a_4;

\fill[black] (2,0) circle (2pt);
\fill[black] (2.1,0) circle (2pt);
\fill[black] (2.2,0) circle (2pt);
\fill[black] (2,-1.5) circle (2pt);
\fill[black] (3,-1.5) circle (2pt);
\fill[black] (3.1,-1.5) circle (2pt);

\braid [number of strands=3,
style strands={1,2}{opacity=0}, height=1 cm, width=0.9cm]at (7.6,0) a_2^{-1};
\braid [number of strands=3,
style strands={1,2}{opacity=0}, height=1 cm, width=0.9cm]at (7.5,0) a_2^{-1};
\braid [number of strands=5,
style strands={1,2,4,5}{opacity=0}, height=1 cm, width=1 cm]at (7.5,0)  a_4^{-1};

\fill[black] (9.3,0) circle (2pt);
\fill[black] (9.4,0) circle (2pt);
\fill[black] (9.5,0) circle (2pt);

\fill[black] (8.4,-1.5) circle (2pt);
\fill[black] (8.5,-1.5) circle (2pt);
\fill[black] (9.5,-1.5) circle (2pt);

\draw  node at (4.25,-0.75) [right]   {$ R^{3}_{12} = \pi_{pt1*} \circ  i_{\mathbb{P}^{2}}^! $};

\draw  node at (10.25,-0.75) [right]   {$R^{3}_{21} = \pi_{pt2*} \circ  i_{\mathbb{P}^{2\vee}}^! $};

\end{tikzpicture}}
            \caption{Third and fourth merges } 
    \end{figure}
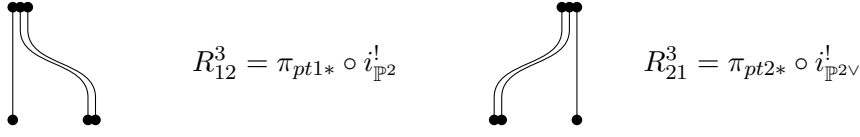

\begin{remark} \label{segal}
The functor $F_{3}^{12}$ (and similarly for the functor $F_{3}^{21}$) can be also described as the functor that maps the 1-dimensional vector space to $\mathcal{O}_{\mathbb{P}^2}$ and its right adjoint $R_{12}^{3}$ is the functor $\mathrm{RHom}(\mathcal{O}_{\mathbb{P}^2},-)$.
\end{remark}

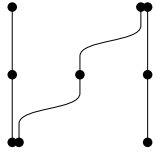
\begin{figure}[!h] \centering 
\adjustbox{scale=0.9,center}{ 
\begin{tikzpicture}

\braid [number of strands=3,
style strands={1,2}{opacity=0}, height=0.5 cm, width=0.9 cm]at (10.1,0) a_2^{-1};
\braid [number of strands=5,
style strands={2,4,5}{opacity=0}, height=0.5 cm, width=1 cm]at (10,0)  a_4;

\braid [number of strands=3,style strands={1,3}{opacity=0}, height=0.5 cm, width=0.90cm]at (10.1,-1) a_1^{-1};
\braid [number of strands=5, style strands={2,4,5}{opacity=0}, height=0.5 cm, width=1 cm]at (10,-1)  a_4;

\fill[black] (10,0) circle (2pt);
\fill[black] (11.9,0) circle (2pt);
\fill[black] (12,0) circle (2pt);

\fill[black] (10,-1) circle (2pt);
\fill[black] (11,-1) circle (2pt);
\fill[black] (12,-1) circle (2pt);

\fill[black] (10,-2) circle (2pt);
\fill[black] (10.1,-2) circle (2pt);
\fill[black] (12,-2) circle (2pt);

\draw  node at (13,-1.5) [right]   {$F_{21}^{111} \text{ is represented by the Fourier-Mukai kernel } \mathcal{O}_{E \times _E D}$.};
\draw  node at (13,-0.5) [right]   {$R_{111}^{12} \text{ is represented by the Fourier-Mukai kernel } \mathcal{O}_{B \times_A A}(B,0)[-1]$.};
\end{tikzpicture}}

\end{figure}

\begin{proposition}\label{naive}
The Fourier-Mukai kernel associated to $R^{12}_{111} F^{111}_{21} $ is $\mathcal{O}_{Z_1}\otimes \mathcal{V}_1^* \otimes (\Lambda ^2 \mathcal{V}_2)^2[-1]$, where  \begin{align*}
Z_1:= \left\{
\xymatrix@C=0.1em{
0 &\subset   & V_2 \ar_{\alpha}@/_2em/[ll]& \subset & \mathbb{C}^3  \ar_{\alpha}@/_2em/[ll] }  \right\} 
\times
\left\{
\xymatrix@C=0.1em{
0 &\subset   & V_1 \ar_{\alpha}@/_2em/[ll]& \subset & \mathbb{C}^3  \ar_{\alpha}@/_2em/[ll] } \right\}
\end{align*}
is a line bundle over $\Fl_3$ and it is also the blow-up of $A$ or $E$ along their zero sections.
\end{proposition}

\begin{proof}
Let  $\pi_{12}, \pi_{23}, \pi_{13}$ be the natural projections
\begin{equation*}
\begin{tikzcd}
  &  E \times C \times A \arrow[dl, "\pi_{12}"' ]\arrow[d, "\pi_{13}" ]\arrow[dr, "\pi_{23}" ]  &\\
E\times C    &  E\times A    &  C\times A  \\
\end{tikzcd}.
\end{equation*}

By (\ref{FM-comp}), the composition of the functors is represented by the convolution of their Fourier-Mukai kernels, thus the kernel of $R^{12}_{111} F^{111}_{21}$ is isomorphic to
\[  \pi_{13*}( \pi_{12}^* \mathcal{O}_{E \times_E D}  \otimes  \pi_{23}^*  \mathcal{O}_{B \times_A A}(B,0)[-1] )= \pi_{13*}( \mathcal{O}_{ E\times_E D \times A}   \otimes     \mathcal{O}_{E \times B\times_A A }(0,B,0) )[-1] \]

The subvariety $E\times_E D \times A$ is of codimension 5 inside $E \times C \times A$; the same codimension is the one of $E \times B\times_A A$ in $E \times C \times A$.

Their intersection $(E\times_E D \times A)\cap ( E \times B\times_A A)=E \times _E (D\cap B) \times _A A$ is smooth and of codimension 10, therefore $E\times_E D \times A$ and $E \times B\times_A A$ intersect transversally in $E \times C \times A$.

By (\ref{transverse}), 
\[ \mathcal{O}_{ E\times_E D \times A}   \otimes     \mathcal{O}_{E \times B\times_A A } \simeq \mathcal{O}_{E \times _E (D\cap B) \times _A A} \]

The subvariety $E \times _E (D\cap B) \times _A A$ can be described as the space
 \begin{align*}
 \left\{
\xymatrix@C=0.1em{
0 &\subset   & V_2 \ar_{\alpha}@/_2em/[ll]& \subset & \mathbb{C}^3  \ar_{\alpha}@/_2em/[ll] }  \right\} 
\times
\left\{
\xymatrix@C=0.1em{
0 &\subset V_1 \subset  & V_2 \ar_{\alpha}@/_2em/[ll]& \subset & \mathbb{C}^3  \ar_{\alpha}@/_-2em/[lll] }  \right\}
\times
\left\{
\xymatrix@C=0.1em{
0 &\subset   & V_1 \ar_{\alpha}@/_2em/[ll]& \subset & \mathbb{C}^3  \ar_{\alpha}@/_2em/[ll] }  \right\}.
\end{align*}

It is clear that $\pi_{13}:E \times _E (D\cap B) \times _A A \to  E\times A$ is an embedding. 

The image $Z_1$ of $E \times _E (D\cap B) \times _A A$ under $\pi_{13}$ is of the form

 \begin{align*}
Z_1= \left\{
\xymatrix@C=0.1em{
0 &\subset   & V_2 \ar_{\alpha}@/_2em/[ll]& \subset & \mathbb{C}^3  \ar_{\alpha}@/_2em/[ll] }  \right\} 
\times
\left\{
\xymatrix@C=0.1em{
0 &\subset   & V_1 \ar_{\alpha}@/_2em/[ll]& \subset & \mathbb{C}^3  \ar_{\alpha}@/_2em/[ll] } \right\}.
\end{align*} 

Notice\footnote{A general argument can be found in section 4.1 of \cite{KT}.} that $\mathcal{O}_{Z_1}(0,B,0)\simeq \mathcal{V}_1^* \otimes (\Lambda ^2 \mathcal{V}_2)^2$, so the Fourier-Mukai kernel of $R^{12}_{111} F^{111}_{21} $ is isomorphic to

\[   \mathcal{O}_{Z_1}\otimes \mathcal{V}_1^* \otimes (\Lambda ^2 \mathcal{V}_2)^2[-1] .\]
\end{proof}

The single crossings are the autoequivalences $T_i$ that induces the Khovanov-Thomas  braid group action on $D^b(C)$ of  Theorem 4.5 of \cite{KT}.

\begin{remark}
The simple crossing functors $ \brTaaaCa$ $\brTaaaCb$ are 
\begin{equation*}
    \brTaaaCa = \cone\left( F_{21}^{111} R_{111}^{21}[-1] \xrightarrow{\trace} \id_{111} \right), 
\end{equation*}
            
\begin{equation*}           
    \brTaaaCb = \cone\left( F_{12}^{111}R_{111}^{12}[-1] \xrightarrow{\trace} \id_{111} \right), 
\end{equation*}
where $tr$ is the counit of the adjunction.

\end{remark}

\begin{proposition}
The Fourier-Mukai kernel associated to $\brTaaaCa$ is the sheaf 
$$\mathcal{O}_{(C\times_C C) \cup (D \times _E D)}(D,0).$$

The Fourier-Mukai kernel associated to $\brTaaaCa$ is the sheaf $\mathcal{O}_{(C\times_C C) \cup (B \times _A B)}(B,0)$.
\end{proposition}

\begin{proof}
Proposition 4.4 in \cite{KT}.
\end{proof}

\section{A split-$\mathbb{P}^2$ action of $\gbrcat_3$ on $T^*\Fl_3(\bar{i})$: main theorem}
\label{cat_act_flag_final_sec}
In this section, we prove that the assignments of the previous sections satisfy the hypothesis of Theorem \ref{main1} and therefore define a categorical action of $\gbrcat _3$ on  $D^b(T^*\Fl_3(\bar{i}))$

\begin{lemma} \label{main1}
The following isomorphism holds
    \begin{equation*}
   F_{21}^{111} F_{3}^{21} \simeq F_{12}^{111} F_{3}^{12}
    \end{equation*}

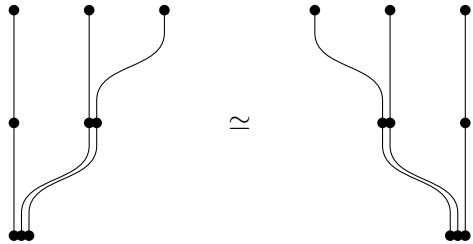
\begin{figure}[!h] \centering    
 \begin{tikzpicture}

\braid [number of strands=3,
style strands={1,3}{opacity=0}, height=1 cm, width=0.9cm]at (0.1,0) a_1^{-1};
\braid [number of strands=3,
style strands={1,3}{opacity=0}, height=1 cm, width=0.9cm]at (0.2,0) a_1^{-1};
\braid [number of strands=5,
style strands={2,3,4,5}{opacity=0}, height=1 cm, width=1 cm]at (0,0)  a_4;

\braid [number of strands=5,
style strands={3,4,5}{opacity=0}, height=1 cm, width=1 cm]at (0,1.5)  a_4;
\braid [number of strands=3,
style strands={1,3}{opacity=0}, height=1 cm, width=0.9cm]at (1.1,1.5) a_1^{-1};

\fill[black] (2,1.5) circle (2pt);
\fill[black] (1,1.5) circle (2pt);
\fill[black] (0,1.5) circle (2pt);
\fill[black] (0,0) circle (2pt);
\fill[black] (0.2,-1.5) circle (2pt);
\fill[black] (0.1,-1.5) circle (2pt);
\fill[black] (0,-1.5) circle (2pt);
\fill[black] (1,0) circle (2pt);
\fill[black] (1.1,0) circle (2pt);

\braid [number of strands=3,
style strands={1,3}{opacity=0}, height=1 cm, width=0.9cm]at (4.1,0) a_2;
\braid [number of strands=3,
style strands={1,3}{opacity=0}, height=1 cm, width=0.9cm]at (4,0) a_2;
\braid [number of strands=3,
style strands={1,3}{opacity=0}, height=1 cm, width=0.9cm]at (3.1,1.5) a_2;

\braid [number of strands=5,
style strands={2,1,4,5}{opacity=0}, height=1 cm, width=1 cm]at (4,0)  a_4;
\braid [number of strands=5,
style strands={1,4,5}{opacity=0}, height=1 cm, width=1 cm]at (4,1.5)  a_4;

\fill[black] (4,1.5) circle (2pt);
\fill[black] (5,1.5) circle (2pt);
\fill[black] (6,1.5) circle (2pt);
\fill[black] (6,0) circle (2pt);
\fill[black] (5.8,-1.5) circle (2pt);
\fill[black] (5.9,-1.5) circle (2pt);
\fill[black] (6,-1.5) circle (2pt);
\fill[black] (5,0) circle (2pt);
\fill[black] (4.9,0) circle (2pt);

\fill[white] (-2,0) circle (2pt);

\draw  node at (3,0)    {$\simeq $};

\end{tikzpicture}
\end{figure}

\end{lemma}

\begin{proof}
The Fourier-Mukai kernel representing $F_{21}^{111} F_{3}^{21} $ is isomorphic to
\begin{equation}
\begin{split}
   &  \pi_{13*} ( \pi_{12}^* ( \mathcal{O}_{pt \times_{pt} A }) \otimes \pi_{23}^* ( \mathcal{O}_{A \times_{A} B }))  \\
    & \simeq \pi_{13*} (  \mathcal{O}_{pt \times_{pt} \mathbb{P}^{2\vee} \times C } \otimes  \mathcal{O}_{pt \times A \times_{A} B }) \\
    & \simeq \pi_{13*} (  \mathcal{O}_{pt \times_{pt} \mathbb{P}^{2\vee} \times_{\mathbb{P}^2} Fl_3 }) \simeq \mathcal{O}_{ pt \times Fl_3} \\
   \end{split}
\end{equation}

Similarly, the Fourier-Mukai kernel associated to $ F_{12}^{111} F_{3}^{12}$ is isomorphic to
\[  \mathcal{O}_{ pt \times Fl_3}.  \]

\end{proof}

\begin{remark}
From Remark \ref{segal}, Lemma \ref{main1} can be also proved showing that functors $F_{21}^{111} F_{3}^{21}$ and  $F_{12}^{111} F_{3}^{12}$ both map the 1-dimensional vector space to $\mathcal{O}_{\Fl_3}$.

\end{remark}

\begin{lemma}\label{main2}
The mapping cone of the morphism 
\[  R_{111}^{12}  F_{21}^{111} R_{111}^{21} F_{12}^{111} \xrightarrow{tr[-2]} \id [-2]\]
is isomorphic to
\begin{equation}
    Cone(tr[-2]) \simeq F_{3}^{12} R_{12}^{3}.
\end{equation}

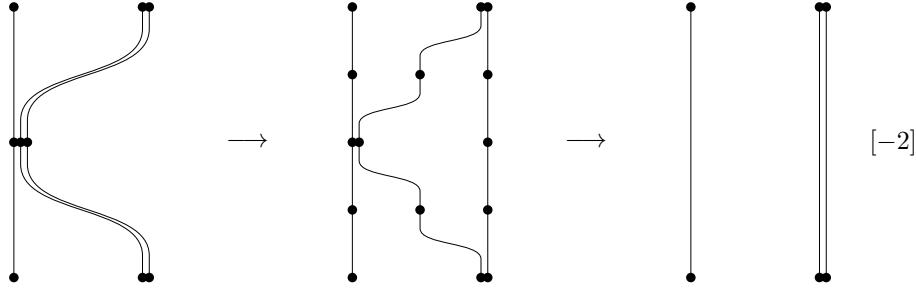
\begin{figure}[!h] \centering    

  \adjustbox{scale=0.9,center}{   \begin{tikzpicture}
\braid [number of strands=3,
style strands={2,3}{opacity=0}, height=1.5 cm, width=1.8cm]at (5.2,-4) a_1;
\braid [number of strands=3,
style strands={2,3}{opacity=0}, height=1.5 cm, width=1.8cm]at (5.1,-4) a_1;
\braid [number of strands=5,
style strands={2,3,4,5}{opacity=0}, height=1.5 cm, width=1 cm]at (5,-4)  a_4;

\fill[white] (3,-2) circle (2pt);

\fill[black] (5.1,-4) circle (2pt);
\fill[black] (5.2,-4) circle (2pt);

  \braid [number of strands=3,
style strands={1,3}{opacity=0}, height=1.5 cm, width=1.8cm]at (5.1,-2) a_1^{-1};
  \braid [number of strands=3,
style strands={1,3}{opacity=0}, height=1.5 cm, width=1.8cm]at (5.2,-2) a_1^{-1};
\braid [number of strands=5,
style strands={2,3,4,5}{opacity=0}, height=1.5 cm, width=2 cm]at (5,-2)  a_4;

\fill[black] (5,-2) circle (2pt);
\fill[black] (6.9,-2) circle (2pt);
\fill[black] (7,-2) circle (2pt);

\fill[black] (5,-4) circle (2pt);

\fill[black] (5,-6) circle (2pt);
\fill[black] (6.9,-6) circle (2pt);
\fill[black] (7,-6) circle (2pt);

\braid [number of strands=3,
style strands={1,2}{opacity=0}, height=0.5 cm, width=0.9cm]at (10.1,-2) a_2^{-1};
\braid [number of strands=5,
style strands={2,4,5}{opacity=0}, height=0.5 cm, width=1 cm]at (10,-2)  a_4^{-1};

\braid [number of strands=3,style strands={3,1}{opacity=0}, height=0.5 cm, width=0.9cm]at (10.1,-3) a_1^{-1};
\braid [number of strands=5, style strands={2,4,5}{opacity=0}, height=0.5 cm, width=1 cm]at (10,-3)  a_4;

\braid [number of strands=3,
style strands={2,3}{opacity=0}, height=0.5 cm, width=0.9cm]at (10.1,-4) a_1;
\braid [number of strands=5,
style strands={2,4,5}{opacity=0}, height=0.5 cm, width=1 cm]at (10,-4)  a_4;

\braid [number of strands=3,style strands={1,3}{opacity=0}, height=0.5 cm, width=0.9cm]at (10.1,-5) a_2;
\braid [number of strands=5, style strands={2,4,5}{opacity=0}, height=0.5 cm, width=1 cm]at (10,-5)  a_4;

\fill[black] (10,-2) circle (2pt);
\fill[black] (11.9,-2) circle (2pt);
\fill[black] (12,-2) circle (2pt);

\fill[black] (10,-3) circle (2pt);
\fill[black] (11,-3) circle (2pt);
\fill[black] (12,-3) circle (2pt);

\fill[black] (10,-4) circle (2pt);
\fill[black] (10.1,-4) circle (2pt);
\fill[black] (12,-4) circle (2pt);

\fill[black] (10,-5) circle (2pt);
\fill[black] (11,-5) circle (2pt);
\fill[black] (12,-5) circle (2pt);

\fill[black] (10,-6) circle (2pt);
\fill[black] (11.9,-6) circle (2pt);
\fill[black] (12,-6) circle (2pt);

\draw  node at (13,-4) [right]   {$\longrightarrow $};
\draw  node at (8,-4) [right]   {$\longrightarrow $};

\braid [number of strands=5,
style strands={2,4,5}{opacity=0}, height=3.5 cm, width=1cm]at (15,-2) a_4;
\braid [number of strands=5,
style strands={1,2,4,5}{opacity=0}, height=3.5 cm, width=1cm]at (14.9,-2) a_4;

\fill[black] (15,-2) circle (2pt);
\fill[black] (16.9,-2) circle (2pt);
\fill[black] (17,-2) circle (2pt);

\fill[black] (15,-6) circle (2pt);
\fill[black] (16.9,-6) circle (2pt);
\fill[black] (17,-6) circle (2pt);

\draw  node at (17.5,-4) [right]   {$[-2] $};

\end{tikzpicture}}

\end{figure}

\end{lemma}

\begin{proof}
The Fourier-Mukai kernel representing  $F_{3}^{12} R_{12}^{3}$ is isomorphic to 
\[ \pi_{13*} ( \pi_{12}^* ( \mathcal{O}_{A \times_{pt} pt }) \otimes \pi_{23}^* ( \mathcal{O}_{pt \times_{pt} A }(0,\mathbb{P}^{2}))[-1] \simeq \] \[ \simeq   \pi_{12}^* ( \mathcal{O}_{\mathbb{P}^{2} \times pt \times A }) \otimes  \mathcal{O}_{A \times pt \times \mathbb{P}^{2}}(0,0,\mathbb{P}^{2}))[-1]  \simeq \]
\[ \pi_{13*} (\mathcal{O}_{\mathbb{P}^{2} \times pt \times  \mathbb{P}^{2}}(0,0,\mathbb{P}^{2})[-1] \simeq \mathcal{O}_{\mathbb{P}^{2} \times  \mathbb{P}^{2}}(0,\mathbb{P}^{2})[-1]. \]

The Fourier-Mukai kernel representing $ R_{111}^{12}  F_{21}^{111} R_{111}^{21} F_{12}^{111} $ is isomorphic, by the proof of Lemma \ref{naive}, to
\[\pi_{15*}(\mathcal{O}_{A \times_A (B\cap D) \times_E E\times C \times A} (0,D,0,0,0) \otimes \mathcal{O}_{A \times C \times E\times_E (B\cap D) \times_A A} (0,0,0,B,0))[-2]. \]

The subvarieties $A \times_A (B\cap D) \times_E E\times C \times A$ and $A \times C \times E\times_E (B\cap D) \times_A A$ are both of codimension 10 in $A\times C \times E \times C \times A$. Their intersection $A \times_A (B\cap D) \times_E E\times  (B\cap D) \times_A A$  is of codimension 20 in  $A\times C \times E \times C \times A$ therefore by (\ref{transverse}) the Fourier-Mukai kernel representing $ R_{111}^{12}  F_{21}^{111} R_{111}^{21} F_{12}^{111}$ is isomorphic to
\[ \pi_{15*}(\mathcal{O}_{A \times_A (B\cap D) \times_E E\times_E (B\cap D) \times_A A} (0,D,0,B,0))[-2] \]
whose support is isomorphic to $(B\cap D) \times_E  (B\cap D)$. 

The variety $(B\cap D) \times_E  (B\cap D)$ has two irreducible components,
one of them $X_1$ is isomorphic to $Z_1$ of Proposition \ref{naive}, and the other one $X_2$ is the blow-up of $\mathbb{P}^2\times \mathbb{P}^2$ along the diagonal.

The intersection $X_1 \cap X_2$ of the two components is isomorphic to $Fl_3$  which is the exceptional divisor inside $X_2$.

Thus, by (\ref{SES}) we have the following short exact sequence
\[ 0 \to  \mathcal{O}_{X_2} (-(X_1\cap X_2)) \to\mathcal{O}_{X_1 \cup X_2} \to \mathcal{O}_{X_1} \to 0  \]
hence we have the isomorphism
\[  \mathcal{O}_{X_1 \cup X_2}\simeq Cone( \mathcal{O}_{X_1}[1] \to  \mathcal{O}_{X_2} (-(X_1\cap X_1)) ) \]

and therefore in $D^b(A \times C \times E \times C \times A)$ the sheaf  $\mathcal{O}_{A \times _A  (B\cap D) \times_E E \times_E (B\cap D) \times_A A}(0,D,0,B,0)$ is isomorphic to
\[   Cone( \mathcal{O}_{Y_1}[1] \to  \mathcal{O}_{Y_2} (-(Y_1\cap Y_1)) ) \otimes \mathcal{O}(0,D,0,B,0). \]

When we project via $\pi_*$ to $A\times A$ the first component $X_1$, it surjects onto the diagonal, while the second component $X_2$ surjects onto $ \mathbb{P}^2 \times \mathbb{P}^2$. 
Since both maps are blow-downs, the projection is an isomorphism except over the diagonal in $\mathbb{P}^2 \times \mathbb{P}^2$ where it is a $\mathbb{P}^1$-bundle.

Thus, taking the derived pushforward $\pi_{15*}$ of $Cone( \mathcal{O}_{Y_1}[1] \to  \mathcal{O}_{Y_2} (-(Y_1\cap Y_1)) ) \otimes \mathcal{O}(0,D,0,B,0)$  and  applying Corollary 4.5 of \cite{AL1} to the map  \[R_{3}^{12} F_{12}^{3} \to R_{111}^{12}  F_{21}^{111} R_{111}^{21} F_{12}^{111} \] we have that
\[ F_{3}^{12} R_{12}^{3} \to R_{111}^{12}  F_{21}^{111} R_{111}^{21} F_{12}^{111} \to \id_{A}[-2]\]
is a distinguished triangle.
\end{proof}

The following Lemma holds in a more general context.

\begin{lemma}
\label{lemma-multifork-isomorphic-to-cohomology-ring-for-flag-varieties}
Let $X$ be smooth projective variety over $k$, $\pi: X \rightarrow \spec k$ be the structure morphism, and $\iota\colon X \hookrightarrow T*(X)$ be the embedding of the zero section:
\begin{equation}
\begin{tikzcd}
 X \ar[hook]{r}{\iota} 
 \ar{d}{\pi}
 &
 T^*(X)
 \\
pt.
\end{tikzcd}
\end{equation}
Let $P^*$ and $P_*$ be the standard Fourier-Mukai kernels of the exact functors 
$$ \pi^* \colon D(pt) \rightarrow D(X), $$
$$ \pi_* \colon D(X) \rightarrow D(pt), $$
and let $I_*$ and $I^!$  be the standard Fourier-Mukai kernels of the exact functors 
$$ \iota_*\colon D(X) \rightarrow D(T^* X), $$
$$ \iota^!\colon D(T^* X) \rightarrow D(X). $$
(See \cite{AL3}, Section 2.6.2 for the details on the standard kernels.) 

Then we have an isomorphism in $D(pt)$:
$$ P_* I^! I_* P^* \simeq \Delta_* H^*(X,k). $$
\end{lemma}

\begin{proof}
By  Lemma 2.19 of \cite{AL3}, we have 
\begin{align}
P_* I^! I_* P^* \simeq (\pi \times \pi)_* I^! I_*. 
\end{align}
By  Proposition 7.8 of \cite{AL3}, the object $I^! I_* \in D(X \times X)$ has the cohomology sheaves:
$$  H^i(I^! I_*) \simeq \Delta_* \wedge^i \mathcal{N}_{X/T^*X}, $$
in degrees $0 \leq i \leq n$ and $0$ elsewhere.

Moreover, by Theorem 1.8(6) of \cite{ACH14},  the object $I^! I_*$ is formal, and hence
$$I^! I_* \simeq \bigoplus_{i = 0}^n \Delta_* \wedge^i \mathcal{N}_{X/T^*X}.$$
Since $\mathcal{N}_{X/T^* X} \simeq \Omega^1_{X/k}$, we conclude that 
$$ I^! I_* \simeq \Delta_* \left( \bigoplus_{i = 0}^n \ \Omega^i_{X/k} \right).$$
Thus we have 
$$ P_* I^! I_* P^* \simeq (\pi \times \pi)_* I^! I_* \simeq  (\pi \times \pi)_* \Delta_* \left( \bigoplus_{i = 0}^n \ \Omega^i_{X/k} \right)
\simeq \Delta_* \pi_* \left( \bigoplus_{i = 0}^n \ \Omega^i_{X/k} \right).$$
Since $\pi_*$ is isomorphic to the derived global section functor $\rder\Gamma(-)$, the assertion of the lemma follows by the degeneration of the Hodge-de-Rham spectral sequence. 
\end{proof}

Putting together the previous Lemmas we have the following:

\begin{theorem} \label{main}
The assignment  $(D^b(\cota \Fl_3 )^\bullet, F^\bullet )$ of:
\begin{enumerate}
    \item the partition $(111)$ to the category  $ D^b(\cota \Fl_3 )$; 
    \item the partition $(12)$ to the category  $D^b(\cota \mathbb{P}^2)$;
    \item the partition $(21)$ to the category  $D^b(\cota \mathbb{P}^{2\vee})$;
    \item the partition $(3)$ to the category  $D^b(pt)$;
\end{enumerate}
and the assignment of:
\begin{enumerate}
    \item the fork $f_{21}^{111}$ to the functor $F_{21}^{111}$;
    \item the fork $f_{12}^{111}$ to the functor $F_{12}^{111}$; 
    \item the fork $f_{3}^{21}$ to the functor $F_{3}^{21}$;
    \item the fork $f_{3}^{12}$ to the functor $F_{3}^{12}$;
\end{enumerate}
define  define a categorical action of $\gbrcat_3$ on $\cota\Fl_3(i)$.
\end{theorem}

\begin{proof}
As in section 7 of \cite{AL3}, $F_{21}^{111}$ and $F_{12}^{111}$ are split spherical functors with cotwist $[-2]$, while $F_{3}^{21}$ and $F_{3}^{12}$ are split $\mathbb{P}^2$ functors with $H=[-2]$.

As a consequence of Lemma \ref{main1}, there exists a multifork isomorphism.

By Lemma \ref{lemma-multifork-isomorphic-to-cohomology-ring-for-flag-varieties}, there exists an isomorphism
        \begin{align*}
            R_{12}^3 R_{111}^{12} F_{12}^{111} F_{3}^{12} \simeq \id_3 \oplus [-2] \oplus [-2] \oplus [-4] \oplus [-4] \oplus [-6]
            \simeq R_{21}^3 R_{111}^{21} F_{21}^{111} F_{3}^{21}
        \end{align*}
and moreover, from Theorem 7.2 of \cite{AL3} , it identifies together with the $\mathbb{P}^2$ functor structure of $F_{3}^{21}$ and $F_{3}^{12}$ the maps \ref{eq1}, \ref{eq2} with

        \begin{equation*}
            \id_3 \oplus [-2] \oplus [-4] \rightarrow \id_3 \oplus [-2] \oplus [-2] \oplus [-4] \oplus [-4] \oplus [-6].
        \end{equation*}

Finally, by Lemma \ref{main2}, the following two diagrams can be completed to two distinguished triangles
 \begin{equation*}
            \label{eqn-skein-relation-no-crossing}
            F_{3}^{12} R_{12}^{3}
            \rightarrow
            R_{111}^{12} F_{21}^{111} R_{111}^{21} F_{12}^{111}
            \rightarrow
            \id_{12}[-2], 
        \end{equation*}

\begin{equation*}
            \label{eqn-skein-relation21-no-crossing}
            F_{3}^{21} R_{21}^{3}
            \rightarrow
            R_{111}^{21} F_{12}^{111} R_{111}^{12} F_{21}^{111}
            \rightarrow
            \id_{21}[-2].
        \end{equation*}

By Theorem \ref{Theorem1} such assignment define a categorical action of $\gbrcat_3$ on $\cota\Fl_3(i)$.
\end{proof}

\begin{remark}
From Remark \ref{gbr2}, a categorical action of $\gbrcat_2$ needs only to satisfy relation (\ref{gbr2relation}).

Therefore, it is a case already covered by Theorem 4.5 of \cite{KT}. 
\end{remark}

\frenchspacing

\newcommand{\etalchar}[1]{$^{#1}$}

\end{document}